\pdfoutput=1
\documentclass[11pt, a4paper, english]{amsart}
\usepackage{amsmath} 
\usepackage{amsthm} 
\usepackage{thmtools, thm-restate}

\usepackage{amssymb} 
\usepackage[dvipsnames]{xcolor}
\usepackage{amscd} 
\usepackage{mathtools}
\usepackage{booktabs}
\usepackage[boxed]{algorithm2e}
\usepackage{subcaption}
\usepackage{centernot}
\usepackage{array} 
\usepackage[cal=boondoxo,scr=euler]{mathalfa}
\usepackage[backref=page,linktocpage]{hyperref} 
\usepackage{cleveref} 
\usepackage{caption} %
\usepackage{graphics,graphicx} 
\usepackage{tikz,tikz-cd} 

\usetikzlibrary{fit,matrix,graphs,graphs.standard,calc,shapes.geometric, arrows.meta}

\usetikzlibrary{decorations.pathreplacing,}

\usepackage{enumerate} 
\usepackage{newtxtext}
\usepackage[varvw]{newtxmath} 

\usepackage{thmtools}
\usepackage{thm-restate}

\DeclareMathAlphabet{\mathsf}{OT1}{\sfdefault}{m}{n}

\newcommand{\nocontentsline}[3]{}
\newcommand{\tocless}[2]{\bgroup\let\addcontentsline=\nocontentsline#1{#2}\egroup}

\usepackage[margin=1.25in]{geometry} 

\usepackage{verbatim}

\usepackage{scalerel}

\makeatletter
\def\dual#1{\expandafter\dual@aux#1\@nil}
\def\dual@aux#1/#2\@nil{\begin{tabular}{@{}c@{}}#1\\#2\end{tabular}}
\makeatother

\makeatletter
\@namedef{subjclassname@2020}{\textup{2020} Mathematics Subject Classification}
\makeatother

\DeclareMathAlphabet{\amathbb}{U}{bbold}{m}{n}

\tikzstyle{rectan} = [rectangle, rounded corners, 
minimum width=1.5cm, 
minimum height=0.75cm,
text width=3cm,
text centered, 
draw=black,
font = \footnotesize,
]

\hypersetup{
    colorlinks = true,
    linkbordercolor = {white},
    linkcolor = {BrickRed},
    anchorcolor = {black},
    citecolor = {BrickRed},
    filecolor = {cyan},
    menucolor = {BrickRed},
    runcolor = {cyan},
    urlcolor = {BrickRed}
}

\usetikzlibrary{automata}

\newtheoremstyle{teoremas}
{12pt}
{13pt}
{\itshape}
{}
{\bfseries}
{}
{.5em}
{}

\theoremstyle{teoremas}
\newtheorem{theorem}{Theorem}[section]
\newtheorem{corollary}[theorem]{Corollary}
\newtheorem{lemma}[theorem]{Lemma}
\newtheorem{proposition}[theorem]{Proposition}

\newtheoremstyle{definition}
{11pt}
{11pt}
{}
{}
{\bfseries}
{}
{.5em}
{}

\theoremstyle{definition}

\newtheorem{conjecture}[theorem]{Conjecture}

\newtheorem{question}[theorem]{Question}
\newtheorem{example}[theorem]{Example}
\newtheorem{remark}[theorem]{Remark}

\crefname{theorem}{theorem}{theorems}
\Crefname{theorem}{Theorem}{Theorems}
\crefname{lemma}{lemma}{lemmas}
\Crefname{lemma}{Lemma}{Lemmas}
\crefname{proposition}{proposition}{propositions}
\Crefname{proposition}{Proposition}{Propositions}

\DeclareMathOperator{\vol}{vol}

\newcommand{\M}{\mathsf{M}}

\renewcommand{\Re}{\operatorname{Re}}
\newcommand{\excise}[1]{}

\title[Examples in Ehrhart theory]{Examples and counterexamples in Ehrhart theory}

\author[L.~Ferroni]{Luis Ferroni}
\author[A.~Higashitani]{Akihiro Higashitani}

\address{(L. Ferroni)
  Department of Mathematics, KTH Royal Institute of Technology, Stockholm, Sweden
}
\email{ferroni@kth.se}

\address{
(A. Higashitani) 
Department of Pure and Applied Mathematics, Graduate School of Information Science and Technology, Osaka University, Osaka, Japan
}
\email{higashitani@ist.osaka-u.ac.jp}

\thanks{The first author is supported by the Swedish
Research Council grant 2018-03968. 
The second author is partially supported by KAKENHI 20K03513 and 21KK0043.}

 \usepackage[notcite,notref]{}
 \usepackage[colorinlistoftodos,bordercolor=orange,backgroundcolor=orange!20,linecolor=orange,textsize=scriptsize]{todonotes}


\allowdisplaybreaks

\begin{document}

\begin{abstract}
    This article provides a comprehensive exposition about inequalities that the coefficients of Ehrhart polynomials and $h^*$-polynomials satisfy under various assumptions. We pay particular attention to the properties of Ehrhart positivity as well as unimodality, log-concavity and real-rootedness for $h^*$-polynomials.

    \noindent We survey inequalities that arise when the polytope has different normality properties. We include statements previously unknown in the Ehrhart theory setting, as well as some original contributions to this topic. We address numerous variations of the conjecture asserting that IDP polytopes have a unimodal $h^*$-polynomial, and construct concrete examples that show that these variations of the conjecture are false. Emphasis is put on polytopes arising within algebraic combinatorics.
    
    \noindent Furthermore, we describe and construct polytopes having pathological properties on their Ehrhart coefficients and roots, and we indicate for the first time a connection between the notions of Ehrhart positivity and $h^*$-real-rootedness. We investigate the log-concavity of the sequence of evaluations of an Ehrhart polynomial at the non-negative integers. We conjecture that IDP polytopes have a log-concave Ehrhart series. Many additional problems and challenges are proposed.
\end{abstract}

\maketitle

\section{Introduction}

Throughout this paper we will focus on lattice polytopes. We will often denote the dimension of a lattice polytope $P\subseteq \mathbb{R}^n$ by $d$. Its \textbf{Ehrhart polynomial}, written $E_P(x)$, is defined for a positive integer $m$ by the map:
    \[ m \mapsto \#(mP \cap \mathbb{Z}^n ).\]
The polynomiality of this function was proved in seminal work of Ehrhart \cite{ehrhart}. For a thorough background on polytopes we refer to \cite{ziegler}, whereas for Ehrhart polynomials we refer to \cite{beck-robins,barvinok}. 

The polynomial $E_P(x)$ encodes fundamental information of the polytope $P$, and can be seen as a vast discrete generalization of the notion of \emph{volume}. It is known that $\deg E_P(x) = d$ and, moreover, writing
    \[ E_P(x) = a_d\,x^d + a_{d-1}\, x^{d-1} + \cdots + a_1 x + a_0,\]
one has $a_d = \vol(P)$, $a_{d-1} = \tfrac{1}{2}\vol(\partial P)$ and $a_0=1$, where $\vol(P)$ denotes the (relative) volume of $P$ and $\vol(\partial P)$ stands for the sum of the (relative) volumes of the facets of $P$ normalized by the facet-defining hyperplane; 
the remaining coefficients are much more complicated to describe \cite{mcmullen}.

The \textbf{$h^*$-polynomial} of the polytope $P$ is defined as the unique polynomial $h^*_P(x)=h_0+h_1x+\cdots+h_dx^d$, having degree at most $d$, and such that:
    \[ \sum_{m=0}^{\infty} E_P(m)\, x^m = \frac{h^*_P(x)}{(1-x)^{d+1}}.\]
The left-hand side is usually referred to as the \textbf{Ehrhart series} of $P$, whereas the vector of coefficients $(h_0,\ldots,h_d)$ is referred to as the \textbf{$h^*$-vector} of $P$. A classical result of Stanley \cite{stanley-hstar} shows that the coefficients of $h^*_P(x)$ are always non-negative integers. The Ehrhart polynomial and the $h^*$-polynomial are related to one another through a change of basis:
    \[ E_P(x) = \sum_{j=0}^d h_j \binom{x+d-j}{d}.\]

The basic inequalities $a_d\geq 0$, $a_{d-1}\geq 0$ and $h_i\geq 0$ for $i=0,\ldots,d$ are of course necessary conditions for a polynomial to be an Ehrhart polynomial. A complete classification of Ehrhart polynomials exists only in dimensions $1$ and $2$ \cite{scott}, whereas for higher dimensions the difficulty of finding a characterization is often assessed as notoriously hard; even in dimension $3$ conjectural inequalities for $h^*$-polynomials are broadly open, see \cite[Conjecture~8.7]{balletti}.  A major quest within discrete geometry consists of finding inequalities that Ehrhart polynomials must satisfy. Quite frequently, it is of interest to study inequalities that $h^*$-polynomials and Ehrhart polynomials must satisfy for \emph{specific} families of polytopes. 

Arguably, within the universe of inequalities that one could possibly consider, the two classes of inequalities that outstand the rest are those known as \emph{Ehrhart positivity} and \emph{$h^*$-unimodality}. A polytope $P$ is said to be \textbf{Ehrhart positive} if all coefficients of $E_P(x)$ are non-negative; for a detailed treatment of Ehrhart positivity we refer to Liu's survey \cite{liu-survey}. A polytope $P$ is said to be \textbf{$h^*$-unimodal} if its coefficients satisfy $h_0\leq \cdots \leq h_{j-1}\leq h_j\geq h_{j+1}\geq\cdots \geq h_d$ for some $0\leq j\leq d$; for a thorough source on $h^*$-unimodality we refer to an article by Braun \cite{braun}. 

Often, to prove that a family of polytopes --- and even quite specific polytopes, are Ehrhart positive or $h^*$-unimodal is a difficult problem. 

Since not all polytopes are Ehrhart positive or $h^*$-unimodal, it is convenient to focus our attention to restricted classes of polytopes possessing additional algebraic or combinatorial structure. For instance, much work in the literature is devoted to proving inequalities that $h^*$-polynomials of polytopes that have the \emph{integer decomposition property (IDP)} satisfy. The precise definition of IDP is stated in Section~\ref{sec:two}. The following is a famous special case of a conjecture appearing in \cite{brenti-update} and often attributed to Stanley; it also appears as a question in \cite[Question~1.1]{schepers-vanlangenhoven}.

\begin{conjecture}\label{conj:intro-idp-implies-unimodal}
    Every lattice polytope satisfying the integer decomposition property (IDP) has a unimodal $h^*$-vector.
\end{conjecture}

A sequence of positive real numbers $(a_0,\ldots,a_d)$ is said to be \textbf{log-concave} if $a_i^2\geq a_{i-1}a_{i+1}$ for every $1\leq i\leq d-1$.\footnote{In this article we will restrict the notion of log-concavity only to sequences of \emph{positive} real numbers, so whenever we say that a sequence is log-concave we are also affirming that its entries are strictly greater than $0$. In some sources such as \cite{brenti}, these sequences are denoted by $\mathrm{PF}_2([x^i])$. Notice that the sequence $(1,0,0,1)$ fulfills the conditions $a_i^2\geq a_{i-1}a_{i+1}$, but will not be considered as log-concave here.} 
A polynomial $p(x)=\sum_{j=0}^d a_j x^j$ is said to be unimodal (resp. log-concave) if the sequence of its coefficients $(a_0,\ldots,a_d)$ is unimodal (resp. log-concave). The following two facts are well-known:
\begin{itemize}
    \item If all roots of $p(x) = \sum_{j=0}^d a_j x^j$ are negative real numbers, then $p(x)$ is log-concave.
    \item If $p(x)$ is log-concave, then $p(x)$ is unimodal.
\end{itemize}
Apart from log-concavity, there is a related notion, that of $\gamma$-positivity, which applies only to palindromic polynomials, but that is also stronger than unimodality and weaker than real-rootedness. A polynomial $p(x) = \sum_{j=0}^d a_j x^j$ of degree $d$ is \emph{palindromic} if $a_j=a_{d-j}$ for each $j=0,\ldots,d$. It is straightforward to see that, in the case of palindromicity, the polynomial $p(x)$ can be written using a different basis, i.e., $p(x) = \sum_{j=0}^{\lfloor\frac{d}{2}\rfloor} \gamma_j\, x^j(x+1)^{d-2j}$. We say that $p(x)$ is \textbf{$\gamma$-positive} if all numbers $\gamma_0,\ldots,\gamma_{\lfloor\frac{d}{2}\rfloor}$ are non-negative. We refer to \cite{athanasiadis-gamma} for a detailed look at $\gamma$-positive polynomials appearing within various combinatorial frameworks.

Similarly, the integer decomposition property admits its own strengthenings and weakenings. For example, it is known that the existence of unimodular triangulations is enough to guarantee IDP-ness. Conjecture~\ref{conj:intro-idp-implies-unimodal}, i.e., ``IDP $\Rightarrow$ $h^*$-unimodality'', makes natural to formulate various similar questions, by either modifying the strength of the assumption or the strength of the conclusion. Moreover, the case of Gorenstein polytopes is of particular interest. It is known by a recent result of Adiprasito, Papadakis, Petrotou and Steinmeyer \cite{adiprasito-papadakis-petrotou-steinmeyer} that Gorenstein IDP polytopes are indeed $h^*$-unimodal. The Gorenstein-ness of a polytope is reflected very neatly in the $h^*$-polynomial, and is actually equivalent to its palindromicity, hence allowing one to formulate questions regarding $\gamma$-positivity as well.

In Section~\ref{sec:three} we will survey the state of the art regarding $h^*$-polynomial inequalities for polytopes that are IDP or have similar properties. In a number of cases, some of the inequalities that we include are known only to few specialists and, to the best of our knowledge, have never been considered explicitly in the Ehrhart theory literature. Also, we will provide examples and counterexamples for numerous assertions related to Conjecture~\ref{conj:intro-idp-implies-unimodal}, both in the general and the Gorenstein (or reflexive) case.

The framework of Ehrhart positivity may seem quite extraneous to people interested in $h^*$-polynomials. In fact, an article by Liu and Solus~\cite{liu-solus} provides several examples that exhibit that there is no general connection between the notion of Ehrhart positivity and $h^*$-unimodality. In Section~\ref{sec:four} we will survey some pathological examples of polytopes failing to be Ehrhart positive, and we will show that in some cases \emph{there is} a connection between Ehrhart positivity and $h^*$-unimodality (actually, $h^*$-real-rootedness). We will establish some results that link these two notions; we believe they can be of good use for combinatorialists. Moreover, we will indicate families of polytopes arising in combinatorics for which we conjecture that our techniques can be applied.

In Section~\ref{sec:five} we will consider a different type of inequalities. The motivating question is the following conjecture.

\begin{conjecture}\label{conj:introidp-ehr-series-logconc}
    If $P$ is an IDP lattice polytope, then the sequence of evaluations of the Ehrhart polynomial at non-negative integers, $E_P(0), E_P(1), E_P(2), \ldots$ is log-concave.
\end{conjecture}

We will study how this log-concavity of the Ehrhart series is related with inequalities for $h^*$-polynomials. For instance, we will show that it is strictly weaker than the $h^*$-polynomial being log-concave, and that it holds no relation at all with $h^*$-unimodality. In particular, one should view Conjecture~\ref{conj:introidp-ehr-series-logconc} as a counterpart of Conjecture~\ref{conj:intro-idp-implies-unimodal} --- it does not imply it, nor is implied by it. We will also investigate the situation if one modifies the assumption of IDP-ness in the statement of our conjecture.

Throughout this article much emphasis will be given to several new problems and conjectures that we will propose. We will collect and summarize all of them in Section~\ref{sec:six}. Since the amount of involved notions is quite extensive, and the way concepts relate to one another might seem confusing for a non-expert, we also included in Figure~\ref{fig:flowchart} a landscape of the implications discussed throughout the paper --- we point out that the dashed arrows correspond to conjectures, whereas the solid arrows are asserted as true. The reader may forgive that, in order to obtain an intelligible planar diagram, we avoided including certain arrows --- e.g., zonotopes are IDP, or Pitman-Stanley polytopes are $\mathcal{Y}$-generalized permutohedra --- and we also omitted the inclusion of certain families of polytopes that are mentioned by passing in the main body of the article.

\begin{figure}[ht]\scalebox{0.80}{
\begin{tikzpicture}
\tikzset{node distance = 1.5cm and 1cm}
\tikzstyle{arrow} = [-{>[scale=1.8,
          length=2,
          width=3.5]},double]

\node (CL) [rectan] {CL-polytope};
\node (magic-basis) [rectan, left =of CL] {Magic positive};
\node (ehr-pos) [rectan, below = of CL] {Ehrhart positive};
\node (ghost) [below = of ehr-pos, xshift = 0.13cm, yshift = 1.5cm] {};
\node (h*-real-root) [rectan, left=of ehr-pos,] {$h^*$-real-rooted};
\node (h*-log-conc) [rectan, below =of h*-real-root,yshift=-1.6cm] {$h^*$-log-concave};
\node (h*-unimodal) [rectan, below =of h*-log-conc] {$h^*$-unimodal};
\node (gamma-pos) [rectan, left =of h*-log-conc] {$h^*$-gamma-positive (if Gorenstein)};
\node (ehr-real-root) [rectan, below =of ehr-pos,yshift=1.2cm] {Ehrhart polynomial real-rooted};

\node (ehr-log-concave) [rectan, right =of h*-log-conc] {Ehrhart series log-concave};

\node (int-ehr-log-concave) [rectan, above =of ehr-log-concave,yshift=-1cm] {Interior Ehrhart series log-concave};

\node (idp) [rectan, right =of h*-unimodal] {IDP};
\node (very-ample) [rectan, right =of idp] {Very ample};
\node (spanning) [rectan, above =of very-ample] {Spanning polytope};
\node (rut) [rectan, below =of idp, xshift=-4cm] {Regular unimodular triangulation};
\node (qt) [rectan, below =of rut,yshift=-2cm] {Quadratic Triangulation};
\node (smooth) [rectan, left =of qt] {Smooth polytope};
\node (zonotope) [rectan, left =of magic-basis] {Zonotope};
\node (pitman-stanley) [rectan, below = of zonotope] {Pitman-Stanley polytope};
\node (compressed) [rectan, above =of smooth, xshift=1.5cm] {Compressed};

\node (alcoved) [rectan, below =of qt] {Alcoved polytope};
\node (y-gen-per) [rectan, right =of rut] {$\mathcal{Y}$-generalized permutohedra};
\node (nestohedra) [rectan, right =of y-gen-per] {Nestohedra};
\node (partial-perm) [rectan, below =of nestohedra] {Partial Permutohedra for $n\geq m-1$};
\node (gen-per) [rectan, right =of qt] {Generalized permutohedra};
\node (polypositroid) [rectan, below =of gen-per] {Polypositroid};
\node (matroid) [rectan, right =of polypositroid] {Matroid};
\node (positroid) [rectan, below =of polypositroid] {Positroid};
\node (shard) [rectan, right =of positroid] {Shard polytope};

\node (order-polytope) [rectan, left =of alcoved] {Order polytope};
\node (sep) [rectan, left =of rut, xshift=-0.25cm] {Symmetric edge polytope};

\draw [arrow] (CL) -- (ehr-pos);
\draw [arrow] (magic-basis) -- (ehr-pos);
\draw [arrow] (magic-basis) -- (h*-real-root);
\draw [arrow] (h*-real-root) -- (gamma-pos);
\draw (ehr-real-root) -- (ehr-pos);
\draw [arrow] (ghost) -| ++(-2.5cm,0) |- (h*-real-root);
\draw [arrow] (h*-real-root) -- (h*-log-conc);
\draw [arrow] (h*-log-conc) -- (h*-unimodal);
\draw [arrow] (h*-log-conc.30) |- (int-ehr-log-concave);
\draw [arrow] (h*-log-conc) -- (ehr-log-concave);
\draw [arrow] (gamma-pos) -- (h*-unimodal);
\draw [arrow,dashed] (idp) -- (h*-unimodal);
\draw [arrow,dashed] (idp) -- (ehr-log-concave);
\draw [arrow] (rut) -- (idp);
\draw [arrow] (idp) -- (very-ample);
\draw [arrow,dashed] (very-ample.120) -- ++(0cm,1cm) -- ++(-2cm,0) |- (int-ehr-log-concave);
\draw [arrow] (very-ample) -- (spanning);
\draw [arrow] (qt) -- (rut);
\draw [arrow] (zonotope) -- (magic-basis);
\draw [arrow] (pitman-stanley) -- (magic-basis);

\draw [arrow,dashed] (qt.110) -- ++(0,1cm) -- ++(-6.5cm,0cm) |- (gamma-pos);
\draw [arrow] (compressed) -- (rut);
\draw [arrow] (alcoved) -- (qt);
\draw [arrow] (y-gen-per) -- (rut);
\draw [arrow] (nestohedra) -- (y-gen-per);
\draw [arrow] (partial-perm) -- (y-gen-per);
\draw [arrow] (sep) -- (rut);
\draw [arrow,dashed] (sep) -- (gamma-pos);
\draw [arrow] (y-gen-per) -- (gen-per);
\draw [arrow] (y-gen-per) -- ++(0cm,1cm) -- ++(6cm,0) |- (ehr-pos.357);
\draw [arrow,dashed] (gen-per) -- (qt);
\draw [arrow] (gen-per.120) -- ++(0cm,1cm) -- ++(-3cm,0cm) -- (rut.334);
\draw [arrow] (alcoved) -- (qt);
\draw [arrow,dashed] (smooth) -- (qt);
\draw [arrow] (polypositroid) -- (alcoved);
\draw [arrow] (polypositroid) -- (gen-per);
\draw [arrow] (matroid) -- (gen-per);
\draw [arrow] (positroid) -- (matroid);
\draw [arrow,dashed] (positroid) -- ++(0,-1cm) -- ++(6.5cm,0) |- (ehr-pos.6);

\draw [arrow] (positroid) -- (polypositroid);
\draw [arrow] (order-polytope) -- (alcoved);
\draw [arrow] (order-polytope) -| ++(-2.8cm,0) |- (gamma-pos.175);
\draw [arrow] (shard) -- (positroid);

\end{tikzpicture}}\caption{A partial view of the landscape (dashed arrows correspond to open conjectures).}\label{fig:flowchart}
\end{figure}
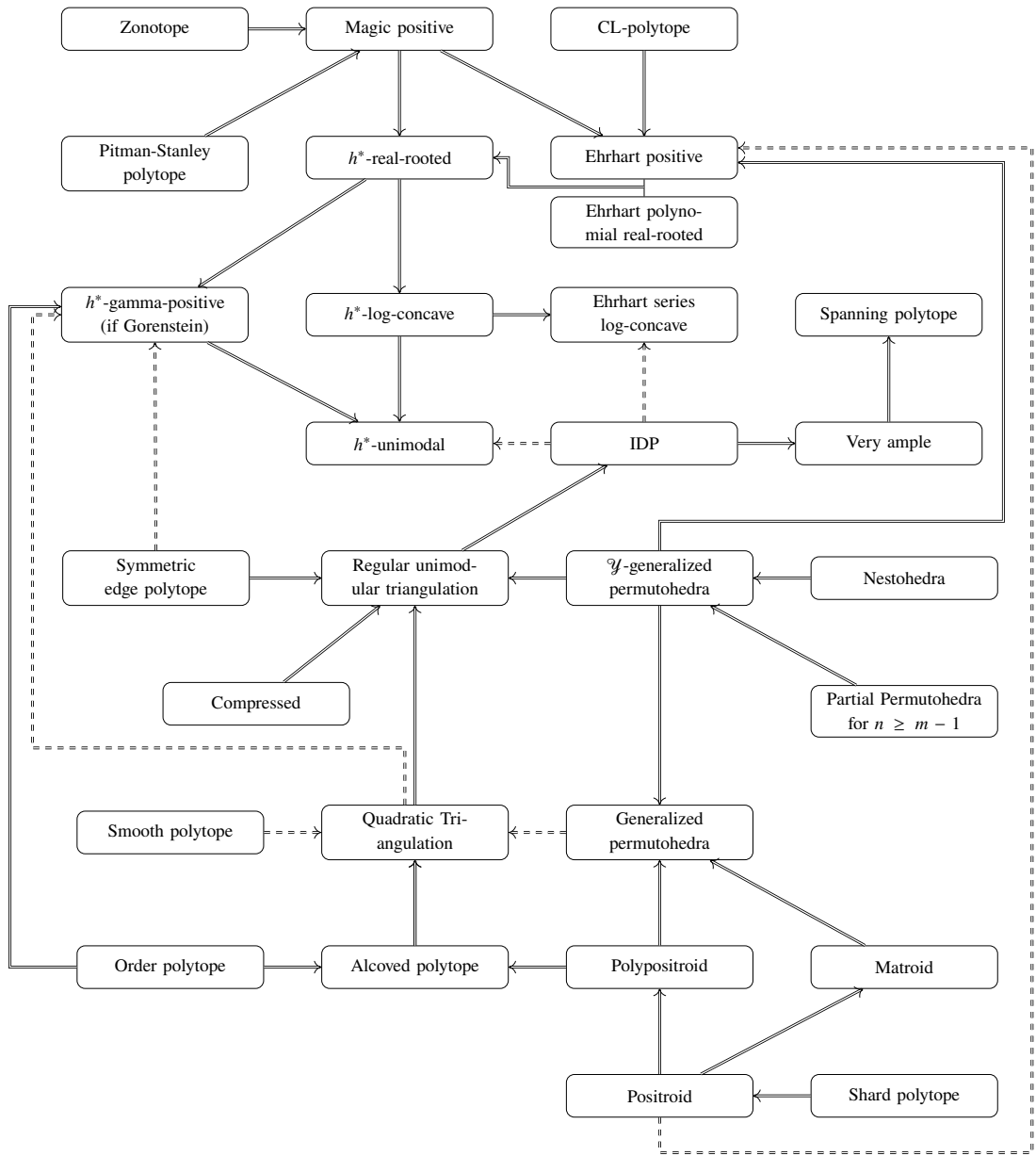

\section{Preliminaries and basic examples}\label{sec:two}

\subsection{Preliminaries in polytopes}

Let us start our discussion with a quick summary of some basic notions that we will refer to throughout the paper. Let us fix a lattice polytope $P\subseteq\mathbb{R}^n$ having dimension $d=\dim P$.

We will say that $P$ has the \emph{integer decomposition property} (IDP)\footnote{We say that $P$ \emph{is} IDP or \emph{has} the IDP interchangeably.} if for every positive integer $k$, every lattice point $p\in kP\cap\mathbb{Z}^n$ can be written as $p = p_1+\cdots+p_k$ where $p_1,\ldots,p_k\in P\cap \mathbb{Z}^n$. If instead of requiring this property for all positive integers $k$, one only requires it for all sufficiently large $k$, the resulting notion is that of \emph{very ample} polytope. If the affine span over $\mathbb{Z}$ of all the points in $P\cap \mathbb{Z}^n$, denoted $\mathrm{aff}_{\mathbb{Z}}(P\cap \mathbb{Z}^n)$, coincides with the lattice $\mathbb{Z}^n$, then we say that $P$ is a \emph{spanning polytope}. 

Let us make a short digression about the terminology. In some sources, e.g. \cite{bruns-gubeladze-book}, a polytope $P$ is said to be ``integrally closed'' when it is IDP, whereas in others, e.g. \cite{cox-little-schenck}, $P$ is said to be ``normal''. The word \emph{normal} is ambiguous in this context, because in \cite{bruns-gubeladze-book} a polytope $P\subseteq \mathbb{R}^n$ is said to be normal if the following equality holds for all positive integers $k$:
\[\underbrace{P \cap \mathbb{Z}^n + \cdots + P \cap \mathbb{Z}^n}_k=kP \cap \mathrm{aff}_{\mathbb{Z}}(P \cap \mathbb{Z}).\] This alternative notion of normality coincides with our definition of IDP when $P$ is a spanning polytope, but not in general.

On the other hand, a \emph{unimodular triangulation} of $P$ is a triangulation of $P$ into simplices having volume $1$. A triangulation is \emph{regular} if it arises via projecting the lower-hull of a lifting of the vertices of the polytope. For a detailed treatment on triangulations, we refer to \cite{deloera-rambau-santos}. A triangulation is said to be \emph{flag} if all of its minimal non-faces have cardinality $2$. A \emph{quadratic triangulation} is by definition a regular unimodular flag triangulation.

The following establishes a hierarchy among these properties.

\begin{theorem}\label{thm:hierarchy}
    The following properties are increasing in strength:
    \begin{enumerate}[\normalfont(i)]
        \item $P$ is spanning.
        \item $P$ is very ample.
        \item $P$ has IDP.
        \item $P$ has a unimodular triangulation.
        \item $P$ has a regular unimodular triangulation.
        \item $P$ has a quadratic triangulation.
    \end{enumerate}
    In other words, if $P$ satisfies one of these properties, then it satisfies all of the previous ones.
\end{theorem}

Most of the implications are immediate by the definitions. The fact that very ample polytopes are spanning follows from the equivalent characterization on very ample polytopes described in \cite[Proposition 2.1]{beck-delgado-gubeladze-michalek}. 
A related list of properties is described in \cite[Section~1.2.5]{haase-paffenholz-piechnik-santos} (see also the references therein); it is worth mentioning that properties such as the existence of \emph{unimodular covers} lies between having the IDP and the existence of unimodular triangulations. The above properties are \emph{strictly} increasing in strength in the sense that one can prove that none of the reverse implications holds in general; we refer to \cite{haase-paffenholz-piechnik-santos} and \cite[Chapter~2.D]{bruns-gubeladze-book} for specific references on examples showing that 
    \[ \text{ IDP \enspace$\centernot{\Longrightarrow}$\enspace Unimodular triangulation \enspace$\centernot{\Longrightarrow}$ \enspace Regular unimodular triangulation},\]
and we refer to \cite{hofscheier-katthan-nill0}, \cite{bruns-quest}, or \cite[Section~9]{balletti} for examples showing that 
    \[ \text{ Spanning \enspace$\centernot{\Longrightarrow}$\enspace Very ample \enspace$\centernot{\Longrightarrow}$ \enspace IDP}.\]
(However, below we will encounter numerous examples of polytopes that certify this second batch of false implications.)

Let us mention another property that is quite important, especially from the toric geometry perspective on polytopes. The polytope $P\subseteq \mathbb{R}^n$ is said to be \emph{smooth} if $P$ is simple and its primitive edge directions at any vertex form a basis of the restricted lattice $\text{aff}(P)\cap \mathbb{Z}^n$. A famous question, known as ``Oda's conjecture'', posed by Oda in \cite{oda}, asks whether all smooth polytopes are IDP. We point out that stronger versions of this conjecture are open as well: for example, it is not known whether all smooth polytopes have regular unimodular triangulations or even quadratic triangulations (see, e.g., \cite[Question~16.6.2]{handbook-dcg} by Lee and Santos and \cite[Question~5.1]{cottonwood} attributed to Knutson).

Let us now state explicitly a result that later will motivate some questions about $h^*$-polynomials. This asserts that all the properties of our hierarchy are preserved under Cartesian products of polytopes.

\begin{proposition}\label{prop:products-ok}
    Let $P$ and $Q$ be lattice polytopes.
    \begin{enumerate}[\normalfont(i)]
        \item If $P$ and $Q$ are spanning, then $P\times Q$ is spanning.
        \item If $P$ and $Q$ are very ample, then $P \times Q$ is very ample. 
        \item If $P$ and $Q$ are IDP, then $P\times Q$ is IDP.
        \item If $P$ and $Q$ have unimodular triangulations, then $P\times Q$ has a unimodular triangulation. 
        \item If $P$ and $Q$ have regular unimodular triangulations, then $P\times Q$ has a regular unimodular triangulation. 
        \item If $P$ and $Q$ have quadratic triangulations, then $P\times Q$ has a quadratic triangulation.
    \end{enumerate}
\end{proposition}

\begin{proof}
    For spanning polytopes, this easily follows from the definition. For very ampleness and IDP, see \cite[Theorem~2.4.7]{cox-little-schenck}. For unimodular, regular unimodular, and quadratic triangulations see \cite[Proposition~2.11]{haase-paffenholz-piechnik-santos} and the discussion thereafter. 
\end{proof}

\subsection{Inequalities for general polytopes}

As was explained in the introduction, a key question in Ehrhart theory is that of finding inequalities that the coefficients of $h^*$-polynomials \emph{must} satisfy. Although we will be mainly concerned with polytopes satisfying additional assumptions (such as being IDP or having a quadratic triangulation), let us mention briefly the state of the art regarding \emph{arbitrary} polytopes. If $P\subseteq\mathbb{R}^n$ is a lattice polytope, we will often denote the dimension of $P$ by $d$ and the degree of the polynomial $h^*_P(x)$ by $s$. 

Let us fix a polytope $P$ having $h^*$-vector $(h_0,\ldots,h_d)$. Arguably, the most elementary inequality that one may derive is the one asserting that $h_1 \geq h_d$. This follows directly from the geometric interpretation of these two coefficients, i.e., $h_1=|P \cap \mathbb{Z}^d| - (d+1)$ and $h_d=|P^\circ \cap \mathbb{Z}^d|$, where $P^\circ$ denotes the relative interior of $P$ (see \cite[Corollary~3.16 and Exercise~4.9]{beck-robins}). 

We recall once again that the most fundamental inequalities assert that the entries of the $h^*$-vector are all non-negative, i.e., $h_i\geq 0$ for $i=0,\ldots,d$. This was proved by Stanley in \cite{stanley-hstar}.

The following result (appearing in this form in \cite{stapledon0}) summarizes several important inequalities that the $h^*$-polynomial of an arbitrary polytope must satisfy.

\begin{theorem}
    Let $P$ be a lattice polytope of dimension $d$ and let $s=\deg h^*_P(x)$. Then, the $h^*$-vector of $P$ satisfies the following inequalities:
    \begin{itemize}
        \item $h_2+h_3+\cdots+h_{i} \geq h_{d-1}+h_{d-2}+\cdots+h_{d-i+1}$ for every $2\leq i\leq \lfloor\frac{d}{2}\rfloor$.
        \item $h_0+h_1+\cdots+h_i \leq h_s+h_{s-1} + \cdots + h_{s-i}$ for every $0\leq i\leq \lfloor \frac{s}{2} \rfloor$.
        \item If $s=d$, then $h_1\leq h_i$ for all $1\leq i\leq d-1$.
        \item If $s\leq d-1$, then $h_0+h_1 \leq h_i + h_{i-1} + \cdots + h_{i-(d-s)}$ for all $1\leq i\leq d-1$.
    \end{itemize}
\end{theorem}

The first set of inequalities is mentioned without proof by Hibi in \cite[Remark~1.4(c)]{hibi-lbt} and proved in \cite{stapledon0}. The second one was first shown by Stanley in \cite[Proposition~4.1]{stanley-idp-ineq}. The third is commonly known as ``Hibi's lower bound theorem'' and is the main result in \cite{hibi-lbt}. The last one is due to Stapledon \cite{stapledon0}. 

By building on a remarkable connection with additive number theory, Stapledon \cite{stapledon} was able to improve upon all these inequalities. He produced an infinite list of (linear) inequalities that $h^*$-vectors must satisfy. For instance, he showed that for every polytope of dimension $d\geq 7$, the inequality $h_{d-1}+2h_{d-2}+h_{d-3}\leq h_2+2h_3+h_4$ holds (this cannot be derived by combining the previously known inequalities). Since the detailed description of these inequalities would require a substantial interleave about additive number theory, we just refer to \cite{stapledon}.

\subsection{Warm-up examples}

Throughout the remainder of the article we will often need to construct examples of polytopes that behave pathologically in terms of certain Ehrhart theoretic properties. The following are some basic examples that we will often combine to produce our desired constructions. In spite of their simplicity, they are of interest in their own right.

\begin{example}\label{ex:reeve-tetrahedra}
    Reeve tetrahedra. Consider the polytope in $\mathbb{R}^3$ having vertices $(0,0,0)$, $(1,0,0)$, $(0,1,0)$ and $(1,q,q+1)$ for $q\in \mathbb{Z}_{\geq 0}$. This polytope will be denoted by $\mathcal{R}_q$. Its Ehrhart polynomial and $h^*$-polynomial are given by:
        \begin{align*} 
        E_{\mathcal{R}_q}(x) &= \tfrac{q+1}{6} x^3 + x^2 + \tfrac{11-q}{6} x + 1\\
        h^*_{\mathcal{R}_q}(x) &= qx^2 + 1. 
        \end{align*}
    If $q=12$, then the Ehrhart polynomial has a negative coefficient. This is a smallest example having such property. 
    Moreover, if $q=34$, then the Ehrhart polynomial is real-rooted yet two of the roots are positive rational numbers. 
\end{example}

\begin{example}\label{ex:reeve-simplices}
    Generalized Reeve simplices. Let $d=2k-1$ for $k\geq 2$, and consider the $d$-dimensional simplex in $\mathbb{R}^{d}$ defined as 
    \[\mathscr{R}_{q,k}:=\mathrm{conv}\{0, e_1, \ldots, e_{d-1}, (1,\ldots,1,q,\ldots,q,q+1)\} \subseteq \mathbb{R}^{d},\] 
    where in the last vertex there are $k-1$ coordinates equal to $1$ followed by $k-1$ coordinates equal to  $q$. When $k=2$ these coincide with the Reeve tetrahedra, i.e., $\mathscr{R}_{q,2}=\mathscr{R}_q$ for every $q$. We have
    \[ h_{\mathscr{R}_{q,k}}^*(x) = q x^k+1.\]
    (for a proof of this fact, we refer the reader to \cite{batyrev-hofscheier}).
\end{example}

\begin{remark}
    If one replaces the fourth vertex in the definition of a Reeve tetrahedron by $(1,p,q+1)$, where $p$ is a positive integer coprime with $q+1$, then the resulting family of polytopes characterizes $3$-dimensional \emph{empty} simplices\footnote{An empty simplex is a simplex which contains no lattice point except for its vertices.}. This is known as \textit{White's theorem} \cite{white64}. Note that a $d$-dimensional simplex having $h^*$-polynomial $h_0+h_1x+\cdots+h_dx^d$ is empty if and only if $h_1=h_d=0$. In particular, the $h^*$-polynomials of $3$-dimensional empty simplices are all of the form $1+qx^2$. Although it is quite hard to characterize higher-dimensional empty simplices, a characterization of $d$-dimensional empty simplices with $d$ odd whose $h^*$-polynomials look like $1+q\,x^{\frac{d+1}{2}}$ is provided by Batyrev and Hofscheier, see \cite{batyrev-hofscheier}. 
\end{remark}

\begin{example}\label{ex:reflexive-simplices}
    The standard reflexive simplex of dimension $d$. This is defined by 
    \[\mathrm{conv}(e_1,\ldots,e_d,-(e_1+\cdots+e_d)).\]
    The $h^*$-polynomial is $x^d+\cdots+x+1$. 
    This is log-concave, but is not $\gamma$-positive. Moreover, this polytope has a regular unimodular triangulation. In fact, it can be obtained by lifting the origin (which is the only lattice point except for the vertices) and leave the vertices as they are. One can easily check that this triangulation is not flag, and hence not quadratic.
\end{example}

In general, for a polytope $P\subseteq \mathbb{R}^n$ with vertices $v_1,\ldots,v_s$ the \emph{pyramid over $P$} is the polytope $\mathrm{Pyr}(P)\subseteq\mathbb{R}^{n+1}$ having as vertices $(v_1,0),\ldots,(v_s,0)$ and $e_{n+1}$. The operation of taking pyramids does not change the $h^*$-polynomial of the polytope, i.e., $h^*_P(x) = h^*_{\mathrm{Pyr}(P)}(x)$ for every polytope $P$ (see \cite[Theorem~2.4]{beck-robins}). At the level of Ehrhart polynomials, the following identity holds for every positive integer $m$:
    \[ E_{\mathrm{Pyr}(P)}(m) = \sum_{j=0}^m E_P(j).\]

\begin{example}\label{ex:pyramid-hypercube}
    The pyramid over a hypercube (cf. \cite[Theorem~2.5]{beck-robins}). Let us denote $P_d=\mathrm{Pyr}([0,1]^d)$. The Ehrhart polynomial has the property that:
    \[ E_{P_d}(m) = \sum_{j=0}^m (j+1)^d\]
    for every positive integer $m$. The polynomial appearing on the right-hand side corresponds to the $d$-th \emph{Faulhaber polynomial} evaluated at $m+1$. 
\end{example}

\section{Inequalities on \texorpdfstring{$h^*$}{h*}-polynomials}\label{sec:three}

The purpose of this section is to summarize the state of the art regarding certain conjectures (and known results) that relate a normality/integrality property of the polytope with inequalities for the entries of its $h^*$-vector. We will also include counterexamples to reasonable conjectures regarding the shape of the $h^*$-vector.

The guiding assertion is a conjecture posed by Stanley in a very influential survey written more than three decades ago.

\begin{restatable}[{\cite[Conjecture~4(a)]{stanley-conj}}]{conjecture}{stanleyconj}
    Let $R = \bigoplus_{i \geq 0} R^i$ be a graded noetherian Cohen-Macaulay domain, defined over a field $\mathbb{F}$. Let us denote $H_R(i):= \dim_{\mathbb{F}} R^i$, the Hilbert function of $R$. If $R$ is generated in degree $1$ and has Krull dimension $d$, then the $h$-polynomial, defined as the unique polynomial $h(x)$ satisfying:
        \[ \sum_{i \geq 0} H_R(i) \, x^i = \frac{h(x)}{(1-x)^d}\]
    has log-concave coefficients.
\end{restatable}

As it was stated, the above conjecture was disproved by Niesi and Robbiano in \cite[Example~2.3]{niesi-robbiano}, where they showed that the $h$-vector $(1,3,5,4,4,1)$ is achievable, and clearly fails to be log-concave. A decade ago, in \cite[Example~2.2]{li-yong} Li and Yong showed that the smaller $h$-vector $(1,2,1,1)$ is also feasible. However, a weaker version of the conjecture, which would amount to asserting the \emph{unimodality} of the $h$-vector remains open (it was explicitly conjectured by Brenti in \cite[Conjecture~5.1]{brenti-update}). Within the realm of Ehrhart polynomials of lattice polytopes, this version of the conjecture, often attributed to Stanley as well, reads as follows.

\begin{restatable}{conjecture}{idpunimodal}\label{conj:idp-implies-unimodal}
    Every lattice polytope satisfying the integer decomposition property (IDP) has a unimodal $h^*$-vector.
\end{restatable}

In what follows we will analyze several variants of this conjecture by varying either the assumption of IDP-ness and/or the conclusion of unimodality. While preparing this section, we realized the importance of shedding more light on the situation regarding certain basic operations on polytopes, in particular Cartesian products, and how they interact with the inequalities for $h^*$-polynomials. 

\subsection{The interplay with Cartesian products}

We will aim to have a better understanding of the behavior of $h^*$-polynomials when a polytope $P$ falls into each of the categories of polytopes in Theorem~\ref{thm:hierarchy}. As Proposition~\ref{prop:products-ok} guarantees that the properties in our hierarchy are preserved under taking products, it is natural to ask whether unimodality, log-concavity, and real-rootedness of $h^*$-polynomials are properties that are preserved under taking products as well. For $h^*$-real-rootedness, we know that the answer is affirmative.

\begin{theorem}[\cite{wagner}]\label{thm:wagner}
    Let $P$ and $Q$ be lattice polytopes. If $h^*_P(x)$ and $h^*_Q(x)$ are both real-rooted polynomials, then $h^*_{P\times Q}(x)$ is real-rooted too.
\end{theorem}

Although innocent-looking, the subtlety of this theorem can hardly be overestimated. It can be obtained as a consequence of a slightly more general result of Wagner \cite[Theorem~0.2]{wagner}, that we shall describe now. In his paper, Wagner defines a (non-linear) operator $\mathcal{W}:\mathbb{R}[x]\to \mathbb{R}[x]$. For a polynomial $f(x)\in \mathbb{R}[x]$, one defines $(\mathcal{W}f)(x)$ as the unique polynomial such that $\deg \mathcal{W}f \leq \deg f$ and:
    \begin{equation}\label{eq:wagner}
    \sum_{m=0}^{\infty} f(m)\, x^m = \frac{(\mathcal{W}f)(x)}{(1-x)^{1+\deg f}}.
    \end{equation}
In Ehrhart theory this is nothing but the map that assigns to an Ehrhart polynomial its corresponding $h^*$-polynomial. Wagner proves that if $\mathscr{W}f$ and $\mathscr{W}g$ are polynomials all of whose roots are negative real numbers, then $\mathscr{W}(fg)$ is negative real-rooted as well.

Turning back our attention to the Ehrhart theory world, it is natural to formulate openly the following questions.\footnote{Part (b) of Question~\ref{question:prods} was answered negatively by Balletti in \cite{balletti-counterexample} a few months after the present article was posted on the arXiv. His example involves a $52$-dimensional polytope.}

\begin{restatable}{question}{products}\label{question:prods}
    Let $P$ and $Q$ be two lattice polytopes.
    \begin{enumerate}[(a)]
        \item If $h^*_P(x)$ and $h^*_Q(x)$ are log-concave, is necessarily $h^*_{P\times Q}(x)$ log-concave too?
        \item If $h^*_P(x)$ and $h^*_Q(x)$ are unimodal, is necessarily $h^*_{P\times Q}(x)$ unimodal too?
    \end{enumerate}
\end{restatable}

We encourage the reader not to underestimate the difficulty of these two challenges. We believe that the answer to the first question might be positive, but the answer to the second is likely negative.  

Our positive belief in the first is that, by making use of the machinery of Lorentzian polynomials of Br\"and\'en and Huh \cite{branden-huh}, it is possible to show that the corresponding statement for ``ultra log-concavity'' is true. The proof of this result was found by Petter Br\"and\'en, and appears in \cite{branden-ferroni-jochemko}. Let us mention that the proof is more delicate than that of Theorem~\ref{thm:wagner}, and can also be adapted to give an alternative proof of it. 

The belief against the second comes from the following reasoning. By paraphrasing the statement of Theorem~\ref{thm:wagner}, one could hope that if $\mathscr{W}f$ and $\mathscr{W}g$ are log-concave (resp. unimodal) then $\mathscr{W}(fg)$ should be log-concave (resp. unimodal) as well. However, the following example answers negatively this broader version of the question in the case of unimodality.

\begin{example}
    Consider the polynomials
        \begin{align*} 
        f(x)&=\tfrac{11}{120} \, x^{5} - \tfrac{7}{24} \, x^{4} + \tfrac{67}{24} \, x^{3} - \tfrac{5}{24} \, x^{2} + \tfrac{217}{60} \, x + 1,\\
        g(x) &= \tfrac{11}{120} \, x^{5} - \tfrac{1}{4} \, x^{4} + \tfrac{61}{24} \, x^{3} + \tfrac{1}{4} \, x^{2} + \tfrac{101}{30} \, x + 1.
        \end{align*}
    We have, correspondingly, $(\mathcal{W}f)(x) = 1+x+x^2+x^3+x^4+6x^5$ and $(\mathcal{W}g)(x) = 1+x+x^2+x^3+2x^4+5x^5$. We see that both $\mathcal{W}f$ and $\mathcal{W}g$ are unimodal. However, we have 
    \begin{align*}
        \mathcal{W}(fg)(x) &= 1+ 38\,x+ 300\,x^2+ 962\,x^3+ 2059\,x^4+ \\
        &\qquad 7442\,x^5+ \boxed{7194}\,x^6+ 7292\,x^7+ 4320\,x^8+ 854\,x^9+ 30\,x^{10},
    \end{align*}
    which fails to be unimodal. Observe, however, that neither $f(x)$ nor $g(x)$ are \emph{Ehrhart polynomials}, as their second coefficients are both negative. We have checked with a computer that this is \emph{the simplest} counterexample that one can construct, in the sense that it minimizes the sum of the coefficients of $\mathscr{W}f$ and $\mathscr{W}g$. This comment intends to give the reader a hint of the real difficulty of finding a counterexample for this question under the Ehrhart-polynomiality assumption. We thank Peter Taylor from MathOverflow for suggesting a slight variation of these examples.
\end{example}

\begin{remark}
    We take the opportunity to point out explicitly that the Minkowski sum of two $h^*$-real-rooted polytopes is not necessarily real-rooted. Let $P$ and $Q$ be the polytopes in $\mathbb{R}^4$ having as vertices the columns of the following two matrices respectively:
    \[
    \begin{bmatrix} 
    0 & 1 & 0 & 0 & 0\\
    0 & 0 & 1 & 0 & 0\\
    0 & 0 & 0 & 1 & 0\\
    0 & 0 & 0 & 0 & 2
    \end{bmatrix}\qquad\qquad
    \qquad\qquad
    \begin{bmatrix} 
    0 & -2\\
    0 & -2\\
    0 & -2\\
    0 & -2
    \end{bmatrix}
    \]
    We have that their $h^*$-polynomials are $h^*_P(x)=h^*_Q(x)=x+1$, and hence are real-rooted. On the other hand, we have 
    \[h^*_{P+Q}(x)=4\,x^4+20\,x^3+20\,x^2+13\,x+1,\]
    which has two non-real roots near $-0.4885\pm 0.69666i$.
\end{remark}

\subsection{Integrality properties and inequalities}

Throughout different attempts to prove or disprove Conjecture~\ref{conj:idp-implies-unimodal}, instead of addressing the various versions of statements asserting either unimodality, log-concavity, $\gamma$-positive, or real-rootedness, a recurring theme within the discrete geometry framework has been that of strengthening or weakening the IDP assumption, and studying the shape of the possible $h^*$-vectors that can arise. Throughout the remaining of this subsection we will collect negative and positive results.

\subsubsection{Don't get too optimistic}

By looking at the hierarchy of Theorem~\ref{thm:hierarchy}, it is clear that the most aspiring strengthening of Conjecture~\ref{conj:idp-implies-unimodal} would be asserting that spanning polytopes have a real-rooted $h^*$-polynomial. In what follows we will see that such statement, as well as other less ambitious strengthenings of Stanley's conjecture turn out to be false.

As a starter, we have that the $h^*$-unimodality may fail at the bottom of the hierarchy.

\begin{theorem}\label{spanning_nonunimodal}
    There exists a spanning polytope whose $h^*$-polynomial is not unimodal.
\end{theorem}

\begin{proof}
    Consider the $5$-dimensional simplex whose vertices are given by the columns of the following matrix:
    \[ \begin{bmatrix} 
    0 & 1 & 0 & 0 & 0 & 5\\
    0 & 0 & 1 & 0 & 0 & 5\\
    0 & 0 & 0 & 1 & 0 & 5\\
    0 & 0 & 0 & 0 & 1 & 5\\
    0 & 0 & 0 & 0 & 0 & 8
    \end{bmatrix}.\]
    Its only interior lattice point is $(2,2,2,2,3)$. Hence, it is immediate to check that it is a spanning polytope. However, its $h^*$-polynomial is given by $h^*(x) = 1 + x + 2x^2 + x^3 + 2x^4 + x^5$ which is not unimodal.
    This example was found by Hofscheier, Katth\"an and Nill in \cite[Example~1.5]{hofscheier-katthan-nill0}.
\end{proof}

The preceding statement reveals that the behavior of the $h^*$-polynomials is \emph{a~priori} not so nice. If one looks at the stronger property of having log-concave $h^*$-polynomials, it turns out that it is possible to produce counterexamples within the class of very ample polytopes.

\begin{theorem}\label{thm:very-ample-not-log-concave}
    There exists a very ample polytope whose $h^*$-polynomial is not log-concave.
\end{theorem}
\begin{proof}
    The following example appears in \cite[Section~9]{balletti}. 
    Consider the $3$-dimensional polytope in $\mathbb{R}^3$ whose vertices are given by the columns of the following matrix:
    \[ \begin{bmatrix} 
    0 & 1 & 0 & 0 & 1 & 0 & 1 & 1 \\
    0 & 0 & 1 & 0 & 0 & 1 & 1 & 1 \\
    0 & 0 & 0 & 1 & 1 & 1 & 16 & 17
    \end{bmatrix}.\]
    This is a segmental fibration (as defined in \cite{beck-delgado-gubeladze-michalek}) over the edge polytope of a cycle on $4$ vertices. The very ampleness of these polytopes is also indicated in \cite[Example~17]{bogart-et-al}, where it is attributed to Bruns. The $h^*$-polynomial of this polytope is $17x^2+4x+1$, which fails to be log-concave.
\end{proof}

For an alternative description of another (higher dimensional) example of a very ample non-IDP polytope with a non-log-concave $h^*$-polynomial, we refer the reader to the Oberwolfach report \cite[p.~2697]{mini-workshop}. In \cite{lason-michalek} Laso\'n and Micha{\l}ek  constructed a family of polytopes generalizing these examples. 

The preceding counterexamples show that unimodality fails already at the level of spanning polytopes, whereas log-concavity fails at the level of very ample polytopes. It is not known whether these two statements can be improved. In other words, we do not know the answer to the following two questions:

\begin{restatable}{question}{veryample}\label{question-very-ample-log-concave}\leavevmode
    \begin{enumerate}[(a)]
        \item Does there exist a very ample polytope whose $h^*$-polynomial is not unimodal? 
        \item Does there exist an IDP polytope whose $h^*$-polynomial is not log-concave?
    \end{enumerate}
\end{restatable}

We mention that Micha{l}ek posed these two questions in \cite{mini-workshop}. Notice that if the answer to any of these two questions is ``no'', then Conjecture~\ref{conj:idp-implies-unimodal} follows immediately. 

The preceding discussion explains at which level of the hierarchy it makes sense to consider the properties of unimodality and log-concavity for $h^*$-polynomials. Regarding $h^*$-real-rootedness, the following example shows that none of the properties in the hierarchy is itself strong enough to guarantee that real-rootedness holds. 

\begin{theorem}\label{thm:quadratic-h-star-not-real-rooted}
    There exists a polytope $P$ having a quadratic triangulation whose $h^*$-polynomial is not real-rooted.
\end{theorem}

\begin{proof}
    Let us call $P$ the order polytope (see Section~\ref{sec:six}) of the poset depicted in Figure~\ref{fig:stembridge-poset}. The polytope $P$ possesses a quadratic triangulation: in fact, the regular unimodular triangulation for order polytopes described by Stanley in \cite[Section~5]{stanley-poset} is flag (alternatively, see \cite[Theorem~3.3]{haase-paffenholz-piechnik-santos}). Its $h^*$-polynomial is given by:
        \begin{align}\label{eq:Stembridge} h^*_P(x) = 3 \, x^{8} + 86 \, x^{7} + 658 \, x^{6} + 1946 \, x^{5} + 2534 \, x^{4} + 1420 \, x^{3} + 336 \, x^{2} + 32 \, x + 1,\end{align}
    which has two non-real roots near $-1.8588 \pm 0.1497i$.
    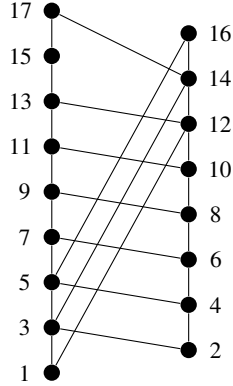
\begin{figure}[ht]
    \begin{tikzpicture}  
		[scale=0.6,auto=center,every node/.style={circle,scale=0.8, fill=black, inner sep=2.7pt}] 
	    \tikzstyle{edges} = [thick];
		\node[label=left:{$1$}] (a1) at (-1,-4) {};  
		\node[label=left:$3$] (a3) at (-1,-3)  {}; 
		\node[label=left:$5$] (a5) at (-1,-2)  {};  
		\node[label=left:$7$] (a7) at (-1,-1) {};  
		\node[label=left:$9$] (a9) at (-1,0)  {};
            \node[label=left:$11$] (a11) at (-1,1)  {};
		\node[label=left:$13$] (a13) at (-1,2)  {};
		\node[label=left:$15$] (a15) at (-1,3)  {};
            \node[label=left:$17$] (a17) at (-1,4)  {};
		
		\node[label=right:$2$] (a2) at (2,-3.5)  {};
		  \node[label=right:$4$] (a4) at (2,-2.5)  {};
		  \node[label=right:$6$] (a6) at (2,-1.5)  {};
		  \node[label=right:$8$] (a8) at (2,-0.5)  {};
		  \node[label=right:$10$] (a10) at (2,0.5)  {};
		  \node[label=right:$12$] (a12) at (2,1.5)  {};
		  \node[label=right:$14$] (a14) at (2,2.5)  {};
            \node[label=right:$16$] (a16) at (2,3.5)  {};

            \draw (a1) -- (a3);
            \draw (a1) -- (a12);
            \draw (a2) -- (a3);
            \draw (a2) -- (a4);
            \draw (a3) -- (a5);
            \draw (a3) -- (a14);
            \draw (a4) -- (a5);
            \draw (a4) -- (a6);
            \draw (a5) -- (a7);
            \draw (a5) -- (a16);
            \draw (a6) -- (a7);
            \draw (a6) -- (a8);
            \draw (a7) -- (a9);
            \draw (a8) -- (a9);
            \draw (a8) -- (a10);
            \draw (a9) -- (a11);
            \draw (a10) -- (a11);
            \draw (a10) -- (a12);
            \draw (a11) -- (a13);
            \draw (a12) -- (a13);
            \draw (a12) -- (a14);
            \draw (a13) -- (a15);
            \draw (a14) -- (a16);
            \draw (a14) -- (a17);
            \draw (a15) -- (a17);
		\end{tikzpicture} \caption{Stembridge's poset}\label{fig:stembridge-poset}
		\end{figure}
\end{proof}

The above proof has a story of its own. The poset of Figure~\ref{fig:stembridge-poset} was used by Stembridge \cite{stembridge} to disprove the Neggers--Stanley conjecture for naturally labelled posets (for arbitrary labellings it had been previously disproved by Br\"and\'en in \cite{branden-neggers}). Later we will see that we can actually strengthen the preceding result so to prove it under the assumption of reflexivity; since that counterexample relies on this one, we prefer to keep both of them.

\subsubsection{But don't lose hope}

On the positive side, the following result, due to Hofscheier, Katth\"{a}n and Nill provides a strong list of inequalities under the mild assumption of having a spanning polytope.

\begin{theorem}[\cite{hofscheier-katthan-nill1}]\label{thm:spanning}
    Let $P$ be a spanning polytope, then the coefficients of the $h^*$-polynomial satisfy the following inequalities:
    \begin{itemize}
        \item $h_1+\cdots + h_i \leq h_{1+j}+\cdots + h_{i+j}$ whenever $i+j\leq s-1$ (Stanley's inequalities)
        \item $h_i \geq 1$ for every $1\leq i\leq s$ (no internal zeros).
    \end{itemize}
\end{theorem}

We mention that the first of the above two sets of inequalities, under the stronger assumptions of having an IDP polytope, had been proved by Stanley in \cite[Proposition~3.4]{stanley-idp-ineq}. In particular this shows that spanning polytopes satisfy $h_1\leq h_j$ for every $1\leq j\leq s-1$. The second set of inequalities asserts that $h^*$-vectors of spanning polytopes do not have ``internal zeros''; for a different proof of those inequalities we also point the reader to \cite{hofscheier-katthan-nill0}.

Beyond the preceding result, no general statements regarding inequalities for $h^*$-polynomials of very ample polytopes are known. However, for IDP polytopes the situation is substantially better, as we have the following important result whose proof was announced recently by Adiprasito, Papadakis, Petrotou and Steinmeyer \cite[Corollary~2.2]{adiprasito-papadakis-petrotou-steinmeyer}.

\begin{theorem}[\cite{adiprasito-papadakis-petrotou-steinmeyer}]\label{thm:adiprasito-steinmeyer}
    Let $P$ be an IDP polytope of dimension $d$, then the coefficients of the $h^*$-polynomial satisfy the following inequalities:
    \begin{itemize}
        \item $h_{d+1-i} \leq h_i$ for each $1\leq i\leq \lfloor\frac{d+1}{2}\rfloor$ (strong bottom-heaviness).
        \item $h_d \leq h_{d-1} \leq \cdots \leq h_{\lfloor\frac{d}{2}\rfloor}$ (second half unimodality).
    \end{itemize}
\end{theorem}

The progression of ideas behind the preceding statement deserves some comments. The same statement, under the stronger assumption of the existence of a regular unimodular triangulation, was proved by Athanasiadis in \cite[Theorem~1.3]{athanasiadis}. It is possible to adapt Athanasiadis' ideas to obtain the same result for polytopes having arbitrary unimodular triangulations, provided that one relies on the validity of the $g$-conjecture for spheres, which was proved by Adiprasito \cite{adiprasito} and Papadakis and Petrotou \cite{papadakis-petrotou}. The even greater upgrade to all IDP polytopes is undoubtedly an exciting development, and the applicability of the ideas behind the proof deserves further study.

Although it is not known whether IDP polytopes in general yield a unimodal or log-concave $h^*$-polynomial, we can prove by elementary means that the first three entries indeed satisfy log-concavity. 

\begin{proposition}\label{prop:idp-log-conc-three-terms}
    Every IDP polytope $P$ of dimension $d \geq 2$, satisfies $h_1^2\geq h_2 h_0$. 
\end{proposition}

\begin{proof}
    First of all, we know that $E_P(1) \geq d+1$. 
    If $E_P(1)=d+1$, i.e., $P$ is an empty simplex, then the IDP-ness of $P$ implies that $P$ is a unimodular simplex, i.e., $h_0=1$ and $h_i=0$ for $i >0$. Hence, the inequality holds. 
    
    In what follows, let us assume that $E_P(1) \geq d+2$. 
    Since $P$ is IDP, every lattice point in $2P$ can be written as a sum of two lattice points in $P$. In particular, we have $\binom{E_P(1)+1}{2} \geq E_P(2)$. 
    Since $E_P(x)=\sum_{i=0}^d h_i^*\binom{x+d-i}{d}$, we can see that $h_0=1$, $h_1=E_P(1)-(d+1)$ and $h_2=E_P(2)-(d+1)E_P(1)+\binom{d+1}{2}$. Combining all these facts, 
    \begin{align*}
    h_1^2-h_2h_0 &=(E_P(1)-(d+1))^2-\left(E_P(2)-(d+1)E_P(1)+\binom{d+1}{2}\right) \\
    &=E_P(1)^2-E_P(2) - (d+1)E_P(1) + \frac{(d+1)(d+2)}{2} \\
    &\geq E_P(1)^2 - \binom{E_P(1)+1}{2} - (d+1)E_P(1)+\binom{d+2}{2} \\
    &=\frac{E_P(1)-(d+2)}{2}E_P(1)+\binom{d+2}{2} \\
    &\geq 0, 
    \end{align*}
    which proves the result.
\end{proof}

Finally, let us take a look at the top of the hierarchy. The following result asserts a list of non-linear inequalities that $h^*$-polynomials satisfy whenever the polytope has a quadratic triangulation. In particular, we can generalize the inequality $h_1^2\geq h_2$ of Proposition~\ref{prop:idp-log-conc-three-terms} to the set of inequalities $h_ih_j\geq h_{i+j}$. In fact, we can go further beyond by leveraging a result about the shape of Hilbert functions of Koszul algebras.

\begin{theorem}\label{thm:quadratic}
    If $P$ has a quadratic triangulation, then the $h^*$-vector $(h_0,\ldots,h_d)$ satisfies the non-linear inequalities described as follows. Consider the infinite matrix 
       \begin{equation}\label{eq:toeplitz-matrix}
       A = \begin{bmatrix} 
        h_0 & h_1   & h_2   & h_3 & h_4 & \cdots \\
        0 & h_0 & h_1   & h_2 & h_3 & \cdots \\
        0 & 0 & h_0 & h_1 & h_2 & \cdots \\
        0 & 0 & 0 & h_0 & h_1 &\cdots\\
        \vdots & \vdots & \vdots & \vdots & \ddots & \vdots\\
        \end{bmatrix}
        \end{equation}
    For each $s\geq 1$ and each sequence of positive integers $i_1,\ldots,i_s$ the $s\times s$ submatrix consisting of rows indexed $(0,i_1, i_1 + i_2, \ldots, i_1+i_2+\cdots+i_{s-1})$ and columns indexed by $(i_1, i_1+i_2, \ldots ,i_1+i_2+\cdots+i_s)$ has non-negative determinant (rows and columns are indexed starting from $0$). For example, when $s=2$, this says that for each $0\leq i,j\leq d$ with $i+j\leq d$ we have:
        \[ h_i h_j \geq h_{i+j}.\]
\end{theorem}

\begin{proof}
    If $P$ has a quadratic triangulation $\mathcal{T}$, since $\mathcal{T}$ is unimodular, the $h$-vector of $\mathcal{T}$ coincides with the $h^*$-vector of $P$. On the other hand, since $\mathcal{T}$ is flag, by a result of Fr\"oberg \cite{froberg}, the Stanley--Reisner ring of $\mathcal{T}$ is Koszul. A standard fact about Koszul algebras (see Backelin and Fr\"oberg \cite[Theorem~4.e.iv]{backelin-froberg} or \cite[Theorem~4.15(iv)]{reiner-welker}), is that taking quotients by regular sequences of linear forms preserves Koszulness. Hence, given that the Stanley-Reisner ring of $\mathcal{T}$ is Cohen-Macaulay because $\mathcal{T}$ triangulates a polytope, after quotienting by a regular sequence of linear forms, we can produce a Koszul (Artininan) algebra whose Hilbert series matches with the $h$-polynomial of the Stanley-Reisner ring of $\mathcal{T}$, i.e., with $h^*_P(x)$. Since $h^*_P(x)$ is the Hilbert series of a Koszul algebra, we can rely on a deep result proved in \cite[Chapter~7, Theorem~2.1]{polishchuk-positselski} by Polishchuk and Positselski. It asserts the coefficients of the Hilbert series of any Koszul algebra are constrained to satisfy precisely the inequalities described by our statement.
\end{proof}

\begin{remark}
    The matrix in \eqref{eq:toeplitz-matrix} is called the \emph{Toeplitz matrix} of the polynomial $h^*_P(x)$. The above result tells us that \emph{some} of the minors of the Toeplitz matrix in equation~\eqref{eq:toeplitz-matrix} are non-negative. On the other hand, observe that
    \begin{itemize}
        \item $h^*_P(x)$ is log-concave if and only if \emph{all minors of size} $2\times 2$ are non-negative.
        \item $h^*_P(x)$ is real-rooted if and only if \emph{all minors of all sizes} are non-negative.
    \end{itemize}
    The first assertion can be verified directly. The second follows from a classical result in \cite{aissen-schoenberg-whitney} which relates the total positivity of Toeplitz matrices with the property of being a P\'olya frequency sequences (in our setting, this means a real-rooted polynomial). For more details about the connection between P\'olya frequency sequences and Koszul algebras, we refer the reader to \cite[Section~4]{reiner-welker}.
\end{remark}

\subsubsection{Combinatorial polytopes and their hierarchies}

Although the existence of quadratic triangulations is a particularly strong feature, many polytopes arising in combinatorics enjoy this property.

\begin{theorem}
    The following families of polytopes possess a quadratic triangulation:
    \begin{enumerate}[\normalfont(i)]
        \item Alcoved polytopes, in particular, 
        \begin{itemize}
            \item Order polytopes.
            \item Lipschitz polytopes. 
            \item Polypositroids.
            \item Positroids.
            \item Rank $2$ matroid polytopes. 
            \item Shard polytopes.
            \item Hypersimplices.
            \item Pitman--Stanley polytopes.
        \end{itemize}
        \item Edge polytopes 
        \begin{itemize}
            \item of complete multipartite graphs;
            \item that are simple polytopes;
            \item of bipartite graphs which do not contain chordless cycles of length at least six. 
        \end{itemize}
        \item Symmetric edge polytopes of trees and complete bipartite graphs.
    \end{enumerate}
\end{theorem}
We refer to Section~\ref{sec:six} for an abridged account of the families mentioned in the preceding statement, as well as references in which the reader may find proofs. We mention explicitly that the above list (and other lists appearing later in this article) may be biased by the authors' own interests, we encourage the reader to not see them as exhaustive.

An outstanding question regarding the existence of quadratic triangulations was posed by Herzog and Hibi in \cite[p.~241]{herzog-hibi}. Paraphrasing them, we cannot resist the temptation of enunciating the following conjecture.  

\begin{restatable}{conjecture}{genperquadratic}\label{conj:gen-perm-quad-triang}
    If $P$ is an integral generalized permutohedron, then $P$ admits a quadratic triangulation. 
\end{restatable}

Although Herzog and Hibi study discrete polymatroids instead of generalized permutohedra, the relation between these two notions is well-understood. A collection of points $A\subseteq \mathbb{Z}^n_{\geq 0}$ is a \emph{discrete polymatroid} if and only if $A$ is the set of \emph{all} the lattice points in an (integral) generalized permutohedron lying in the closed positive orthant, see \cite[Section~9]{lam-postnikov-polypositroids}. Herzog and Hibi ask whether the toric ideal of a discrete polymatroid admits a quadratic Gr\"obner basis, and it is known that an affirmative answer is equivalent to the existence of a quadratic triangulation for the corresponding polytope, i.e., the validity of Conjecture~\ref{conj:gen-perm-quad-triang}. This is related to another hierarchy of famous conjectures posed by White in \cite{white}; among these, the weakest one posits that the toric ideal of a matroid is quadratically generated --- though, of course, the stronger property of having a quadratic Gr\"obner basis, and thus the existence of a quadratic triangulation for the matroid base polytope are still open problems. Though still open, some progress on White's conjectures has been achieved in \cite{lason-michalek-white,lason,hayase-hibi-katsuno-shibata}.

We note that in the hierarchy of Theorem~\ref{thm:hierarchy}, integral generalized permutohedra are known to be IDP (for a proof, see \cite[Section~3]{mcdiarmid}, or in the case of matroids \cite[Chapter~10]{michalek-sturmfels}). Furthermore, a recent breakthrough by Backman and Liu~\cite[Corollary~3.6]{backman-liu} proves the following.

\begin{theorem}[\cite{backman-liu}]
    If $P$ is an integral generalized permutohedron, then $P$ admits a regular unimodular triangulation.
\end{theorem}

The less restrictive property of having a regular unimodular triangulations occurs even more commonly within combinatorial polytopes. More so, it is conjectured to be true in a number of cases.

\begin{theorem}
    The following families of polytopes admit regular unimodular triangulations.
    \begin{enumerate}[\normalfont(i)]
        \item Alcoved polytopes of type B.
        \item Compressed polytopes (also known as $2$-level polytopes), in particular
        \begin{itemize}
            \item Order polytopes. 
            \item Birkhoff polytopes.
            \item Integral Gelfand-Tsetlin polytopes $P_{\lambda/\mu,\mathbf{1}}$ having weight $1$.
        \end{itemize}
        \item $\mathcal{Y}$-generalized permutohedra, in particular,
            \begin{itemize}
                \item Nestohedra, in particular graph associahedra.
                \item Fertilitopes.
                \item Partial permutohedra $\mathscr{P}(m,n)$ for $n\geq m-1$ 
                \item Pitman-Stanley polytopes.
            \end{itemize}
        \item Symmetric edge polytopes.
    \end{enumerate}
\end{theorem}

We refer to \cite{haase-paffenholz-piechnik-santos} and to Section~\ref{sec:six} for more details about the preceding statement.

\subsection{The Gorenstein and reflexive case}

We now discuss the situation of Gorenstein and reflexive polytopes. In this case, we have a palindromic $h^*$-vector, which in addition allows us to consider the notion of $\gamma$-positivity.

\subsubsection{Negative results}

In \cite{hibi} Hibi had conjectured that all reflexive polytopes had unimodal $h^*$-polynomials. If one relaxes the requirement of reflexivity to the property of being Gorenstein, counterexamples were already known; for instance the Reeve tetrahedron (cf. Example~\ref{ex:reeve-tetrahedra}) for $q=1$ is already an example. In \cite{mustata-payne} Musta\c{t}\v{a} and Payne found a counterexample to Hibi's conjecture. After that, Payne \cite{payne} found additional such examples. We borrow \cite[Example~4.2]{payne} in order to formulate the following statement.

\begin{theorem}
    There is a spanning reflexive polytope of dimension $d$ with non-unimodal $h^*$-polynomial.
\end{theorem}

\begin{proof}
    Though we have already constructed an example in Theorem \ref{spanning_nonunimodal}, we can also construct a similar example with a reflexive polytope. 
    In fact, let us consider the $6$-dimensional simplex whose vertices are given by the columns of the following matrix: 
    \[ \begin{bmatrix} 
    1 & 0 & 0 & 0 & 0 & 0 & -1\\
    0 & 1 & 0 & 0 & 0 & 0 & -1\\
    0 & 0 & 1 & 0 & 0 & 0 & -1\\
    0 & 0 & 0 & 1 & 0 & 0 & -1\\
    0 & 0 & 0 & 0 & 1 & 0 & -1\\
    0 & 0 & 0 & 0 & 0 & 1 & -3
    \end{bmatrix}.\]
    Then it is trivially spanning by the first six columns and has its $h^*$-polynomial given by $1+x+2x^2+x^3+2x^4+x^5+x^6$. 
\end{proof}

\begin{remark}
    In fact, Gorenstein polytopes can have \emph{highly non-unimodal} $h^*$-vectors. Put precisely, for each $k \geq 2$ and $\ell \geq 2$, set $d=k\ell-1$, and consider the $d$-dimensional polytope:
    \begin{align*}
        P:=\mathrm{conv}(\{{\bf 0}\} \cup \{e_i+\cdots+e_{i+k-1} : i=1,\ldots,k\ell-1\}) \subseteq \mathbb{R}^d,  
    \end{align*}
    where $e_{d+i}=e_i$ for each $i$. One can prove that: 
    \[h_P^*(x)=1+x^\ell+x^{2\ell}+\cdots+x^{(k-1)\ell}.\]
    For more about this construction we refer to \cite[Section~3]{higashitani}.
\end{remark}

As we will explain later, the existence of a quadratic triangulation for a Gorenstein polytope $P$ is expected to reflect into the $h^*$-polynomial being $\gamma$-positive (see Conjecture~\ref{conj:ehrhart-gal} below). As the following example shows, regular unimodular triangulations are not enough to guarantee $\gamma$-positivity.

\begin{theorem}\label{thm:reflexiv-unimod-triang-not-gamma-positive}
    There exists a reflexive polytope $P$ having a regular unimodular triangulation but whose $h^*$-polynomial is not $\gamma$-positive.
\end{theorem}

\begin{proof}
    Consider the reflexive simplex of Example~\ref{ex:reflexive-simplices}. Its $h^*$-polynomial is $h_P^*(x) = 1+x+x^2+\cdots+x^d$. This polynomial is not $\gamma$-positive for $d\geq 2$.
\end{proof}

For general polynomials with positive coefficients, it is often asserted that $\gamma$-positivity and log-concavity are properties that do not compare to one another. Within the Ehrhart theory world, it is easier to construct polynomials that are log-concave but not $\gamma$-positive (for instance, having many coefficients equal to the same number), than to construct $h^*$-polynomials being $\gamma$-positive but not log-concave. The following statement provides both such examples.

\begin{theorem}
    There exist Gorenstein polytopes whose $h^*$-polynomials are $\gamma$-positive and not log-concave, and vice-versa.
\end{theorem}

\begin{proof}
    An example of a lattice polytope whose $h^*$-polynomial is log-concave but not $\gamma$-positive is a standard reflexive simplex, cf. Example~\ref{ex:reflexive-simplices}. For a polytope whose $h^*$-polynomial is $\gamma$-positive but not log-concave, consider the polytope $P$ whose vertices are the columns of the matrix:
    \[\begin{bmatrix}
        1 & 0 & 0 & 1 & -9\\
        0 & 1 & 0 & 1 & -5\\
        0 & 0 & 1 & 1 & -3\\
        0 & 0 & 0 & 2 & -2
    \end{bmatrix}\]
    One has $h^*_P(x)=1+4x+22x^2+4x^3+x^4$ which is a $\gamma$-positive polynomial failing to be log-concave, since $16=h_1^2<h_0h_2=22$. 
\end{proof}

The following result is a strengthening of Theorem~\ref{thm:quadratic-h-star-not-real-rooted}. It shows that even in the reflexive case, the existence of quadratic triangulations is not strong enough to guarantee the real-rootedness of the $h^*$-polynomial.

\begin{theorem}\label{thm:non-realrooted}
    There exists a Gorenstein (in fact, reflexive) polytope having a quadratic triangulation, whose $h^*$-polynomial is not real-rooted.
\end{theorem}

\begin{proof}
    Recently, Ohsugi and Tsuchiya initiated the study of a new family of polytopes: the class of enriched chain polytopes \cite{ohsugi-tsuchiya-2020-1}. This class arises by taking a poset $L$ on $[n]$ and defining the \emph{enriched chain polytope of $L$}, as the convex hull of the set 
    \[\{\pm e_{i_1} \pm \cdots \pm e_{i_k} : \{i_1,\ldots,i_k\} \text{ is an antichain of }L\}.\]  

    Let us call $Q$ the enriched chain polytope arising from Stembridge's poset (see Figure~\ref{fig:stembridge-poset}). In \cite[Example 3.7]{ohsugi-tsuchiya-2020-2}, the polytope $Q$ is presented as a ``symmetric edge polytope of type B'' (where the graph is precisely the complement of the comparability graph of Stembridge's poset).

    By relying on a formula for $h^*$-polynomials of symmetric edge polytopes of type B proved in \cite[Theorem 0.3]{ohsugi-tsuchiya-2020-1}, it is possible to compute the $h^*$-polynomial of $Q$:
    \begin{align*}
        h^*_Q(x)&=(x+1)^{17}h^*_P\left(\frac{4x}{(x+1)^2}\right) \\
        &=x^{17} + 145x^{16} + 7432x^{15} + 174888x^{14} + 2128332x^{13}+14547884x^{12} + 59233240x^{11} \\
        &\qquad+148792184x^{10} + 234916470x^9 + 234916470x^8 + 148792184x^7 +59233240x^6 \\
        &\qquad+ 14547884x^5 + 2128332x^4 + 174888x^3+7432x^2 + 145x + 1,
    \end{align*}
    where $P$ is the order polytope of Stembridge's poset  --- its $h^*$-polynomial was explicitly stated in equation~\eqref{eq:Stembridge}.
    This palindromic log-concave polynomial has two non-real roots near $-3.88091 \pm 0.18448i$. 
    A technical note is that, even though the ``left peak polynomial'' $W_P^{(\ell)}$ should be used in general, since the poset under consideration is narrow (i.e. it allows a partition into two chains), we can replace it by $W(P)(x)$, the $P$-Eulerian polynomial, which coincides with $h^*_P(x)$; see \cite[Remark 2.3]{ohsugi-tsuchiya-2020-1}. The fact that this polytope possesses a quadratic triangulation follows from \cite[Proposition~3.8]{ohsugi-tsuchiya-2020-2}.
\end{proof}

\subsubsection{Positive results}

Let us undertake the discussion initiated above. After Musta\c{t}\v{a} and Payne \cite{mustata-payne} found counterexamples to the initial conjecture posed by Hibi in \cite{hibi}, a refinement of the conjecture was proposed by Hibi and Ohsugi \cite{hibi-ohsugi}. This refined version of the conjecture asserts that Gorenstein IDP polytopes are $h^*$-unimodal. The most important positive result regarding inequalities for Gorenstein and reflexive polytopes is precisely a proof of Hibi and Ohsugi's conjecture, achieved by Adiprasito, Papadakis, Petrotou and Steinmeyer.

\begin{theorem}[\cite{adiprasito-papadakis-petrotou-steinmeyer}]
    If $P$ is a Gorenstein IDP polytope, then $h^*_P(x)$ is unimodal.
\end{theorem}

We point out that in \cite[Theorem~1.3]{adiprasito-papadakis-petrotou-steinmeyer} Adiprasito, Papadakis, Petrotou and Steinmeyer refer only to reflexive polytopes (as opposed to Gorenstein polytopes). However, a result of Bruns and R\"omer \cite[Corollary~7]{bruns-romer} guarantees that the set of $h^*$-polynomials of IDP reflexive polytopes is the same as the set of $h^*$-polynomials of IDP Gorenstein polytopes\footnote{This is not true if one removes the assumption of IDP-ness. For example, the Reeve tetrahedron of Example~\ref{ex:reeve-tetrahedra} with $q=1$ is Gorenstein, with $h^*$-polynomial equal to $x^2+1$; this cannot be the $h^*$-polynomial of a reflexive polytope by the main result of \cite{hibi-lbt}.}. 
It is worth mentioning that the above result under the stronger assumption of the existence of regular unimodular triangulation was proved precisely by Bruns and R\"omer \cite{bruns-romer} --- and, as they explain in their paper, their argument would work as well for \emph{arbitrary} unimodular triangulations, provided that the $g$-conjecture for spheres holds; this has been proved recently in \cite{adiprasito,papadakis-petrotou}. Previous to the work of Bruns and R\"omer, by building upon ideas of Reiner and Welker \cite{reiner-welker} and motivated by the problem of proving the $h^*$-unimodality of Birkhoff polytopes, Athanasiadis  \cite{athanasiadis-birkhoff} showed that if $P$ is compressed and Gorenstein then $h^*_P(x)$ is unimodal.

It is natural to ask whether the preceding result can be upgraded to log-concavity. As the following proposition shows, whenever the degree of the polytope is not too big, log-concavity holds indeed.

\begin{proposition}\label{prop:d leq 5}
    If $P$ is a Gorenstein IDP polytope with $\deg h^*_P(x) \leq 5$, then $h^*_P(x)$ is log-concave.
\end{proposition}

\begin{proof}
    Let $h^*_P(x)$ be the $h^*$-polynomial and let $s=\deg h_P^*(x) \leq 5$. We split the proof into three cases: 
    \begin{itemize}
    \item If $s\leq 3$, since $h_1 \geq h_0$ holds and $h_P^*(x)$ is palindromic, then we conclude the assertion without IDP-ness. 
    
    \item If $s=4$, we have $h_P^*(x) = 1 + a x + bx^2 + a x^3 + x^4$ for some non-negative integers $a,b$. 
    By Proposition~\ref{prop:idp-log-conc-three-terms}, we know that $a^2 \geq b$. 
    Moreover, since $P$ is spanning if $P$ is IDP, it follows from Theorem~\ref{thm:spanning} that $h_1 \leq h_2$ holds by taking $i=j=1$, i.e., $b \geq a$ holds, which implies that $b^2 \geq a^2$. 
    Hence, the log-concavity of $h_P^*(x)$ holds. 
    
    \item If $s=5$, we have $h^*(x) = 1 + a x + bx^2 + b x^3 + a x^4 + x^5$. 
    By the same argument as in the case $s=4$, we obtain that $h_P^*(x)$ is log-concave. \qedhere
    \end{itemize}
\end{proof}

\begin{remark}\label{rem:unimod}
    For general reflexive polytopes of dimension $d \leq 5$, the IDP-ness is not required to claim $h^*$-unimodality. 
    If $d \leq 5$ and $P$ is reflexive, then $h^*_P(x)$ is unimodal. Moreover, if $d \leq 3$, then $h^*_P(x)$ is log-concave. This directly follows from \cite{hibi-lbt}, which claims that $h_i \geq h_1$ holds for $i=1,\ldots,d-1$ if $P$ contains an interior lattice point. Note that $0 < h_d \leq h_1$ is always true if $P$ contains an interior lattice point.
\end{remark}

One of the most important open problems in topological combinatorics is a conjecture posed by Charney--Davis \cite[Conjecture~D]{charney-davis}. A strengthening proposed by Gal in \cite{gal} asserts the $\gamma$-positivity of the $h$-polynomial of a flag simplicial complex that is homeomorphic to a sphere. The special case of Gal's conjecture involving Ehrhart theory, reads as follows.

\begin{restatable}[Gal's conjecture for $h^*$-polynomials]{conjecture}{galehrhart}\label{conj:ehrhart-gal}
    If $P$ is a Gorenstein polytope having a quadratic triangulation, then the $h^*$-polynomial of $P$ is $\gamma$-positive.
\end{restatable}

As Theorem~\ref{thm:reflexiv-unimod-triang-not-gamma-positive} shows, the flagness of the regular unimodular triangulation is an essential assumption in the above conjecture. We mention that if $P$ is a Gorenstein polytope of degree $s\leq 4$, then the above conjecture holds (in fact, real-rootedness holds); we refer to \cite[Section~3.1]{gal}.

Perhaps the most important special case for which Conjecture~\ref{conj:ehrhart-gal} is known to hold, is for order polytopes, thanks to a result of Br\"and\'en.

\begin{theorem}[\cite{branden-gamma}]\label{thm:branden-gamma}
    The order polytope of a graded poset has a $\gamma$-positive $h^*$-polynomial.
\end{theorem}

For an abridged account of order polytopes we refer to Section~\ref{sec:six}. The preceding statement is a special case of Conjecture~\ref{conj:ehrhart-gal} because an order polytope of a poset $P$ is Gorenstein if and only if the poset is graded, and order polytopes possess quadratic triangulations. The unimodality for the $h^*$-polynomial of order polytopes of graded posets had previously been obtained by Reiner and Welker via the $g$-theorem for simplicial polytopes \cite{reiner-welker}. It is not known whether Br\"and\'en's result can be upgraded to real-rootedness, while by Theorem~\ref{thm:non-realrooted} it is known that there exists a Gorenstein polytope having a quadratic triangulation whose $h^*$-polynomial is not real-rooted. 
In other words, Conjecture~\ref{conj:ehrhart-gal} cannot be strengthened to real-rootedness, whereas it is not known whether Theorem~\ref{thm:branden-gamma} can.

\begin{remark}
In \cite{higashitani-unimodal}, Higashitani constructed many examples concerning unimodality of $h^*$-polynomials. 
In fact, \emph{alternatingly increasing} $h^*$-polynomials are discussed and provide many examples on this, e.g., 
lattice polytopes having non-log-concave alternatingly increasing $h^*$-polynomials, those having neither log-concave nor alternatingly increasing but unimodal $h^*$-polynomials, and so on. 
\end{remark}

\section{Ehrhart positivity}\label{sec:four}

A property that has received considerable attention in the past few years is that of ``Ehrhart positivity''. We will say that a polytope $P$ is \emph{Ehrhart positive} if all coefficients of the Ehrhart polynomial $E_P(x)$ are non-negative. If one writes
    \[ E_P(x) = a_d x^d + a_{d-1} x^{d-1} + \cdots + a_1 x + 1,\]
where $d=\dim P$, then $a_d$, $a_{d-1}$ and $a_0$ are known to be non-negative --- in fact, positive. The reason is that $a_d = \vol(P)$, $a_{d-1}=\frac{1}{2}\vol(\partial P)$ and $a_0=1$. For the remaining coefficients, usually referred to as the \emph{middle coefficients}, the situation is not as nice, because they can in fact be negative.

One of the main sources on Ehrhart positivity is a survey by Liu \cite{liu-survey}. In this section we will make a brief summary of the salient points regarding polynomials either having (or not) this important property. We will discuss several pathological examples, plausible assertions that turn out to be false (or open), and we will further discuss the connection with real-rootedness (and thus log-concavity and unimodality) of $h^*$-polynomials.

\subsection{Examples of polytopes having Ehrhart polynomials with negative coefficients}

In order to have a better understanding of how badly the Ehrhart positivity may fail, we present a couple of examples that reveal that, in general, the signs of the Ehrhart coefficients can behave in a quite unpredictable way. Both of these families of examples were found by Hibi, Higashitani, Tuschiya and Yoshida.

\begin{theorem}[{\cite[Theorem~1.1]{hibi-higashitani-yoshida-tsuchiya}}]
    For every $d \geq 3$, there exists a lattice polytope $P$ such that all coefficients of degree $1,\ldots,d-2$ of $E_P(x)$ are negative.
\end{theorem}

\begin{proof}
    Fix $d\geq 3$ and $q\geq 1$, and consider the $d$-dimensional polytope $P_q:=[0,d-3]^{d-3} \times \mathcal{R}_q$ (cf. Example~\ref{ex:reeve-tetrahedra}). When $q$ is sufficiently large, one can prove that the Ehrhart polynomial:
        \[ E_{P_q}(x) = ((d-3)x+1)^{d-3} \cdot \left(\tfrac{q+1}{6} x^3 + x^2 + \tfrac{11-q}{6} x + 1\right),\]
    has all of its middle coefficients negative.
\end{proof}

\begin{theorem}[{\cite[Theorem~1.2]{hibi-higashitani-yoshida-tsuchiya}}]
    For all $d$ and $k$ with $1 \leq k \leq d-2$, there exists a lattice polytope $P$ such that the coefficient of $x^k$ of $E_P(x)$ is negative and the remaining coefficients are all positive. 
\end{theorem}

The construction of these polytopes is slightly more involved, so we refer the reader directly to \cite[Section~3]{hibi-higashitani-yoshida-tsuchiya} for the details. We point out that an outstanding conjecture by Hibi, Higashitani, Tsuchiya and Yoshida asserts that \emph{arbitrary} negative patterns may arise.

\begin{restatable}[\cite{hibi-higashitani-yoshida-tsuchiya}]{conjecture}{negativepatterns}
    For every $d\geq 3$, and every $I\subseteq \{0,1,\ldots,d\}$ such that $\{0,d-1,d\}\subseteq I$, there exists a lattice polytope $P$ of dimension $d$, having the property that the non-negative Ehrhart coefficients of $P$ are exactly $[x^i]E_P(x)$ for $i\in I$.
\end{restatable}

In addition to the existence of such wild examples having negative Ehrhart coefficients, even very simple operations can fail to preserve Ehrhart positivity. The operations of taking pyramids, or taking Minkowski sums do not preserve Ehrhart positivity. For pyramids, it suffices to consider Example~\ref{ex:pyramid-hypercube}: the hypercube $[0,1]^{20}$ has Ehrhart polynomial $(x+1)^{20}$ whereas the pyramid $P_{20}$ over it has an Ehrhart polynomial with a negative coefficient. For Minkowski sums, we refer to some examples constructed by Tsuchiya and discussed in \cite[Section~4.6]{liu-survey}. 

\subsection{Examples and non-examples of Ehrhart positive families}

In order to motivate further discussion regarding Ehrhart positivity, let us mention some families of polytopes for which Ehrhart positivity is known to hold. Most of these were addressed in Liu's survey \cite{liu-survey}, but we take the opportunity to update the list by adding new families for which now this property has been verified.

\begin{theorem}
    The following families of polytopes are known to be Ehrhart positive.
    \begin{enumerate}[\normalfont(i)]
        \item $\mathcal{Y}$-generalized permutohedra, in particular:
            \begin{itemize}
                \item Partial permutohedra $\mathscr{P}(m,n)$ for $n\geq m-1$.
                \item Nestohedra (in particular, graph associahedra).
                \item Fertilitopes.
            \end{itemize}
        \item Hypersimplices (and similar slices of boxes).
        \item Rank $2$ matroid polytopes.
        \item Zonotopes.
        \item Pitman-Stanley polytopes.
        \item Cyclic polytopes (higher integral polytopes).
    \end{enumerate}
\end{theorem}

We refer to Section~\ref{sec:six} for more details about the above polytopes as well as references for the proofs. The reader should not be misguided by the above list, often it is the case that for small dimensional examples of polytopes belonging to a family of interest, Ehrhart positivity does hold, yet for sufficiently large dimensional polytopes within the family, counterexamples arise. Following this yoga, let us mention in detail a couple of negative results that encompass many others --- e.g., that quadratic triangulations do not guarantee Ehrhart positivity, and other statements of similar flavor. The first one is that matroid polytopes are not Ehrhart positive in general. This was a conjecture posed by De~Loera, Haws and K\"oppe in \cite{deloera-haws-koppe} and disproved in \cite{ferroni3}; we take the opportunity to give an outline of the proof for readers whose background is leaned towards polyhedral geometry rather than matroid theory.

\begin{theorem}[\cite{ferroni3}]\label{thm:matroid-negative}
    There exist matroid polytopes with Ehrhart polynomials having a negative coefficient.
\end{theorem}

\begin{proof}
    Matroid polytopes are by definition the $0/1$-polytopes all of whose edges are of the form $e_i-e_j$. In \cite{ferroni2} it is proved that if $P\subseteq \mathbb{R}^n$ is a matroid polytope and $P'\subseteq \mathbb{R}^n$ is a matroid polytope obtained by ``deleting'' one vertex of $P$, then the Ehrhart polynomials of both $P$ and $P'$ differ by a fixed polynomial. Precisely:
        \[ E_{P}(x) = E_{P'}(x) + D_{k,n}(t-1),\]
    where $D_{k,n}(t)$ is a polynomial depending only of the ambient dimension $n$ and the value $k$ determined by the hyperplane $x_1+\cdots+x_n=k$ in which both $P$ and $P'$ lie. An explicit expression for the polynomial $D_{k,n}(t)$ is obtained in \cite{ferroni2}. In \cite{ferroni3} it is proved that actually for every $k$ and $n$, the hypersimplex $\Delta_{k,n}$ is a matroid polytope admitting the removal of $\lfloor\frac{1}{n}\binom{n}{k}\rfloor$ vertices, one at each step, always preserving the edge directions. For every $k$ and $n$, if one performs such removals, the end product is a matroid polytope $P$ whose Ehrhart polynomial is:
        \[ E_P(x) = E_{\Delta_{k,n}}(x) - \left\lfloor\frac{1}{n}\binom{n}{k}\right\rfloor D_{k,n}(t-1).\]
    Since the Ehrhart polynomial of the hypersimplex can be calculated explicitly \cite{katzman,ferroni1,ferroni-mcginnis}, one can compute in a fast way the polynomial $E_P(x)$. For sufficiently large choices of the parameters $n$ and $k$, one can show that $E_P(x)$ can have negative coefficients. In particular, $n = 4000$ and $k=3$ provides a rank $3$ connected matroid on $4000$ elements and having a negative quadratic Ehrhart coefficient; this is a $3999$-dimensional polytope.
\end{proof}

The second important non-example for Ehrhart positivity is that of order polytopes (and hence, alcoved polytopes). This was studied in detail by Liu and Tsuchiya in \cite{liu-tsuchiya}

\begin{theorem}[\cite{liu-tsuchiya}]\label{thm:alcoved-non-ehr-pos}
    There exist alcoved polytopes (actually, order polytopes) with Ehrhart polynomials having negative coefficients.
\end{theorem}

\begin{proof}
    The pyramid over a hypercube of dimension $n$ (see Example~\ref{ex:pyramid-hypercube}) corresponds to a poset of the form depicted in Figure~\ref{fig:stanley-poset}.
\begin{figure}[ht]
    \centering
	\begin{tikzpicture}  
	[scale=0.7,auto=center,every node/.style={circle,scale=0.8, fill=black, inner sep=2.7pt}] 
	\tikzstyle{edges} = [thick];
	
	\node[] (a1) at (0,0) {};  
	\node[fill=white,scale=1.3] (a2) at (0,2)  {$\qquad\cdots$};  
	\node[] (a4) at (2,2) {};
	\node[] (a5) at (3,2)  {};  
	\node[] (a6) at (-1,2)  {};  
	\node[] (a7) at (-2,2)  {};  
	\node[] (a8) at (-3,2)  {};  
	\node[] (a9) at (-4,2) {};
	
	\draw[edges] (a1) -- (a4);  
	\draw[edges] (a1) -- (a5);  
	\draw[edges] (a1) -- (a6);
	\draw[edges] (a1) -- (a7);  
	\draw[edges] (a1) -- (a8);  
	\draw[edges] (a1) -- (a9);  
	\end{tikzpicture}\caption{One bottom element, covered by $n$ elements.}\label{fig:stanley-poset}
\end{figure}
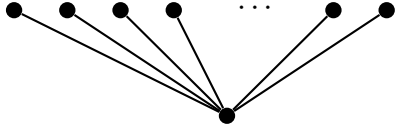

    \noindent
    These posets were described by Stanley in \cite[Exercise~3.164]{stanley-ec1} and further studied by Liu and Tsuchiya in \cite{liu-tsuchiya}. When $n\geq 20$, the Ehrhart polynomials of the order polytopes of these posets have negative coefficients.
\end{proof}

\begin{remark}
    We comment that even smooth polytopes fail to be Ehrhart positive in general. An example of this phenomenon is constructed in \cite{castillo-liu-nill-paffenholz} by Castillo, Liu, Nill and Paffenholz, answering the Ehrhart positivity question raised by Bruns in \cite[Question~7.1]{bruns-quest}.
\end{remark}

Although the two mentioned instances, namely matroid polytopes and alcoved polytopes, do not possess the property of Ehrhart positivity, there is an outstanding conjecture that predicts that it will certainly hold within the \emph{intersection} of these two classes. 

\begin{restatable}[\cite{ferroni-jochemko-schroter}]{conjecture}{positroidsehrhartpositive}\label{conj:positroids-ehr-pos}
    Alcoved matroid polytopes, also known as \emph{positroids}, are Ehrhart positive.
\end{restatable}

Positroid polytopes encompass important classes of polytopes, such as hypersimplices, series-parallel matroid polytopes, shard polytopes \cite{padrol-pilaud-ritter}, rank $2$ matroid polytopes, lattice path matroid polytopes \cite{knauer-martinez-ramirez}, and Schubert matroid polytopes \cite{fan-li}. There is some evidence supporting Conjecture~\ref{conj:positroids-ehr-pos}, including the cases of hypersimplices \cite{ferroni1}, rank $2$ matroid polytopes \cite{ferroni-jochemko-schroter}, and ``minimal'' matroids \cite{ferroni2}. On the other hand, there are some other conjectures related to the above. The \emph{cuspidal} matroids \cite{ferroni-schroter} correspond to polytopes defined by four parameters, $k,n,r,s$ and described as follows:
    \[ P = \left\{ x\in [0,1]^n : \sum_{i=1}^n x_i = k, \; \sum_{i=1}^r x_i \leq s\right\}.\]
These polytopes are special cases of positroids, and they were conjectured to be Ehrhart positive by Fan and Li in \cite[Conjecture~1.6]{fan-li} (they call them ``notched rectangle matroids''). In \cite{hanely-et-al} Hanely et al. address the case $s=1$ of the above family; they named \emph{panhandle matroids} the corresponding subfamily of polytopes. They conjectured the Ehrhart positivity for $s=1$ in \cite[Conjecture~6.1]{hanely-et-al}, and actually propose a concrete combinatorial description of the coefficients. The Ehrhart positivity for $r=2$ and $s=1$ was recently settled by McGinnis in \cite{mcginnis}, via a non-trivial combinatorial interpretation of the Ehrhart coefficients.

Two further families for which Ehrhart positivity has been checked extensively using computers are that of edge polytopes and symmetric edge polytopes. It would be of much interest for the community to find proofs or counterexamples for the following two (a priori unrelated) conjectures.

\begin{restatable}{conjecture}{edgesymedgeehrpos}
    Let $G$ be a graph.
    \begin{enumerate}[\normalfont(i)]
        \item The edge polytope of $G$ is Ehrhart positive.
        \item The symmetric edge polytope of $G$ is Ehrhart positive.
    \end{enumerate}
\end{restatable}

Partial evidence for the first of these two conjectures is a result by Ferroni, Jochemko and Schr\"oter \cite{ferroni-jochemko-schroter} which proves that Ehrhart positivity holds for complete multipartite graphs (these are polytopes of rank $2$ matroids). Notice that edge polytopes are the subpolytopes of second hypersimplices. We have checked that the above conjectures hold for all graphs on at most $9$ vertices by an exhaustive computation.

\subsection{Real-rootedness of Ehrhart polynomials}

A question raised by Postnikov at the MSRI Workshop on Geometric and Topological Combinatorics in 2017 was whether there is any relationship between the positivity of the Ehrhart coefficients with the unimodality of the $h^*$-polynomial. In \cite{liu-solus} Liu and Solus addressed several interesting examples that show that in general there is no relationship between these two notions --- i.e., they provide examples of polytopes having both, exactly one or none of these two properties. 

In spite of the absence of a general connection between the two notions, we will prove some results that are quite useful in many instances, and establish a link between these two notions. The following result is new.

\begin{theorem}\label{thm:ehr-real-rooted-implies-h*-real-rooted}
    If $P$ is Ehrhart positive and the Ehrhart polynomial of $P$ is real-rooted, then the $h^*$-polynomial is real-rooted too (in particular, it is log-concave and unimodal).
\end{theorem}

\begin{proof}
    Denote $E_P(x)$ the Ehrhart polynomial of the $d$-dimensional polytope $P$. By the reciprocity law, for each positive integer $m$, we have that $|E_P(-m)|$ counts the number of lattice points lying in the relative interior of $mP$. In particular, $E_P(-\ell)=0$ for all integers $1\leq \ell \leq \lambda$, where $\lambda$ is the smallest non-negative integer such that $\lambda P$ has no interior lattice points but $(\lambda+1)P$ does. In particular, it is known that this implies that $h_{d}=h_{d-1}=\cdots=h_{d+1-\lambda}=0$ and $h_{d-\lambda}>0$. 

    Let us show now that for all real numbers $u \geq \lambda + 1$, we have that $E_P(-u) \neq 0$. Again, using the reciprocity law we have:
        \[ (-1)^d E_P(-u) = \sum_{i=1}^{d+1}h_{d+1-i}\binom{u+d-i}{d}= \sum_{i=\lambda+1}^{d+1}h_{d+1-i}\binom{u+d-i}{d}.\]
    Notice that the assumption $u\geq \lambda+1$ yields:
    \begin{align*}
        (-1)^d E_P(-u) &= \sum_{i=\lambda+1}^{d+1}h_{d+1-i}\binom{u+d-i}{d}\\
        &= h_{d-\lambda} \binom{u-(\lambda+1)+d}{d} + h_{d-\lambda-1} \binom{u-(\lambda+2)+d}{d} + \cdots + h_0 \binom{u-1}{d}\\
        &\geq h_{d-\lambda} \binom{u-(\lambda+1)+d}{d}\\
        &>0.
    \end{align*}
    It follows that whenever $\lambda$ is the largest integer for which $E_P(-\lambda)=0$, then the largest real number $\rho$ for which $E_P(-\rho)=0$ satisfies $\lambda \leq \rho < \lambda+1$. On the other hand, since our polytope is assumed to be Ehrhart positive, we are guaranteed that all the real roots of $E_P(x)$ are negative real numbers.
    
    In \cite[Theorem~4.4.1]{brenti} Brenti proved that under these assumptions (i.e., that a negative real-rooted polynomial vanishes at all negative integers with absolute value smaller than the ``largest'' negative real root),  
    if we perform the change of basis:
        \[ E_P(x) = \sum_{i=0}^d h_i \binom{x+d-i}{d},\]
    then the sequence $(h_0,h_1,\ldots,h_d)$ is a P\'olya frequency sequence; in turn, by a classical result in \cite{aissen-schoenberg-whitney}, the polynomial $h_0+h_1x+\cdots + h_d x^d$ is real-rooted. Of course, the coefficients $h_0,\ldots, h_d$ arising by the above change of basis match with the $h^*$-vector of $P$.
\end{proof}

\begin{remark}
    The preceding proof relies on the fact that the polynomial $E_P(x)$ is an Ehrhart polynomial in an essential way. In other words, it is not true that if $f(x)$ is a real-rooted polynomial of degree $d$ having positive coefficients, then the polynomial $W(x):=(\mathscr{W}f)(x)$, defined as in equation~\eqref{eq:wagner}, is real-rooted as well. Even if one asks that $W(x)$ has positive coefficients, this fails to be true. For instance, the polynomial $f(x) = \frac{43}{6}x^3+\frac{41}{2}x^2+\frac{46}{3}x+1$ has positive coefficients and is real-rooted whereas, correspondingly, $W(x) = 1+40x+x^2+x^3$ is not. In Brenti's notation \cite{brenti}, this shows that
        \[ \mathrm{PF}[x^i] \cap \mathrm{PF}_1 \left[\binom{x+d-i}{d}\right] \not\subseteq \mathrm{PF} \left[\binom{x+d-i}{d}\right],\]
    so one should view Theorem~\ref{thm:ehr-real-rooted-implies-h*-real-rooted} as a \emph{little Ehrhartian miracle}. In the general polynomial setting, an important result due to Brenti \cite[Theorem~4.4.4]{brenti} says that if $f(x)$ has only real zeros, and all of them lie in the interval $[-1,0]$, then $\mathscr{W}f$ is real-rooted.
\end{remark}

Although Ehrhart positivity does not hold for all polytopes, it happens to be true quite frequently. Therefore, one can use the preceding statement as a tool for discharging the problem of proving the unimodality (or log-concavity, or real-rootedness) of the $h^*$-polynomial to that of the real-rootedness of the Ehrhart polynomial. Before going to examples in which this situation applies, let us stress the fact that the Ehrhart positivity condition is essential, as the following result shows.

\begin{theorem}
    There exists a lattice polytope $P$ such that $E_P(x)$ is real-rooted, yet the $h^*$-polynomial is not real-rooted (in fact, not even unimodal).
\end{theorem}

\begin{proof}
    An example in dimension $3$ is obtained by taking the Reeve tetrahedron, as in Example~\ref{ex:reeve-tetrahedra}, for $q=34$. The Ehrhart polynomial is \[E_{\mathscr{R}_{34}}(x)= \tfrac{35}{6}\,x^3 + x^2-\tfrac{23}{6}\,x+1=\tfrac{35}{6}(x+1)\left(x-\tfrac{3}{7}\right)\left(x-\tfrac{2}{5}\right),\]
    which is real-rooted, but has a negative coefficient. Observe that we have $h_{\mathscr{R}_{34}}^*(x) = 34x^2+1$, which is not unimodal.
\end{proof}

\begin{theorem}
    The following families of polytopes are Ehrhart positive and their Ehrhart polynomials are real-rooted.
    \begin{enumerate}[\normalfont(i)]
        \item Rectangular prisms.
        \item Weakly increasing parking function polytopes $\mathcal{X}_n^w(a,b)$.
    \end{enumerate}
    Hence, their $h^*$-polynomials are real-rooted.
\end{theorem}

A further instance in which this result can be applied is to prove the real-rootedness of the biEulerian polynomial introduced in the work of Ardila; see the proof of \cite[Theorem~6.3]{ardila-bipermutohedron}. On the other hand, notice that the case of rectangular prisms (i.e., products of pairwise orthogonal segments) is of enough interest on its own, as this shows (once again) that the $h^*$-polynomial of the cube $[0,1]^n$ (i.e., the Eulerian polynomial) is real-rooted, as well as the $h^*$-polynomial of the cube $[0,2]^n$ (i.e., the Eulerian polynomial of type B). Various other combinatorial polynomials arise in this way (for more examples, see~\cite[Section~4.9]{savage-schuster}). Furthermore, an example that comprises all the ones discussed in this paragraph are the $r$-colored Eulerian polynomials, whose real-rootedness also follows from Theorem~\ref{thm:ehr-real-rooted-implies-h*-real-rooted}.

Now, let us finish the discussion on roots of Ehrhart polynomials by mentioning certain pathological properties that they may have. Of course, a condition that the Ehrhart polynomial $E_P(x)$ of a lattice polytope $P$ must satisfy is that $1=E_P(0)\leq E_P(1) \leq E_P(2) \leq \ldots\,$. In particular, we observe that positive integers are \emph{never} roots of an Ehrhart polynomial. However, the following statement, which is a modification of another appearing in \cite[Theorem~1.3]{beck-loera-develin-pfeifle-stanley}, shows that positive real roots of Ehrhart polynomials can be quite wild, in the sense that one can have \emph{large} real positive roots and that they can be \emph{very close} to being an integer.

\begin{theorem}
    For every $\varepsilon > 0$, there exists a sufficiently large positive integer $d$ and a polytope $P$ of dimension $d$ having the following property: the Ehrhart polynomial $E_P(x)$ admits real positive roots $r_1 < r_2 < \cdots < r_s$, where each $|r_i - i| < \upvarepsilon$ and $s\approx \frac{d}{2\pi e}$.
\end{theorem}

\begin{proof}
    The polytope $P_d$ described in Example~\ref{ex:pyramid-hypercube} has Ehrhart polynomial
        \[ E_{P_d}(m) = \sum_{j=1}^{m+1} j^d .\]
    The right-hand side is a shifted evaluation of a Faulhaber polynomial. For odd values $d=2n+1$, one has $E_{P_d}(x) = \frac{1}{d} B_{d}(x)$, where $B_d(x)$ denotes the $d$-th \emph{Bernoulli polynomial} (see \cite[Theorem~2.5]{beck-robins}). The Bernoulli polynomials make a prominent appearance in number theory and complex analysis, they are notorious for having a ``sinusoidal'' behavior near the first positive integers. Its complex roots are well understood; the main results of \cite{veselov-ward} imply that for odd values $d=2n+1$, the polynomial $B_d(x)$ has real roots close to the half integers $\frac{3}{2}, \frac{4}{2}, \frac{5}{2},\ldots, \frac{m}{2}$ where $m\approx \frac{d}{2\pi e}$. Moreover, the distance of the roots to these integers approaches $0$ as the dimension $d$ grows.
\end{proof}

\subsection{CL-polytopes}

We now discuss a family of reflexive polytopes whose Ehrhart polynomial satisfies a particular property on its roots. A reflexive polytope $P$ is said to be a \emph{CL-polytope} if all of complex roots of the Ehrhart polynomial $E_P(x)$ lie on the so-called \emph{critical line} $\Re(z) = -\frac{1}{2}$. Univariate polynomials all of whose complex roots have negative real part are usually called \emph{Hurwitz stable} polynomials. Of course, the definition of CL-polytopes implies that their Ehrhart polynomials are Hurwitz stable.

\begin{lemma}
    If $Q(x)$ is a Hurwitz stable polynomial with real coefficients, then the coefficients of $Q(x)$ are all non-negative. In particular, all CL-polytopes are Ehrhart positive.
\end{lemma}

\begin{proof}
    Since $Q(x)$ is assumed to have real coefficients, its roots are either negative real numbers, or complex numbers whose conjugate is also a root. In other words, we have:
    \[ Q(x) = A(x-r_1)\cdots (x-r_a) (x-z_1)(x-\overline{z_1}) \cdots (x-z_b)(x-\overline{z_b}),\]
    where $r_1,\ldots,r_a$ are negative real numbers and $z_1,\ldots, z_b$ are complex numbers having negative real part. The factor $(x-r_1)\cdots (x-r_a)$ is a product of polynomials having positive coefficients, hence it has positive coefficients. Similarly, each factor $(x-z_i)(x-\overline{z_i})$ can be written as $x^2 - 2\Re(z_i) x + |z_i|^2$, and therefore has non-negative coefficients. The result follows.
\end{proof}

\begin{theorem}
    The following families of polytopes are known to be CL-polytopes.
    \begin{itemize}
        \item Standard reflexive simplices. 
        \item Cross polytopes and their dual polytopes $[-1,1]^d$. 
        \item Root polytopes of type A and their dual polytopes. 
        \item Root polytopes of type C and their dual polytopes. 
        \item Dual polytopes of associahedra. 
        \item Symmetric edge polytopes of complete bipartite graphs $K_{a,b}$ with $\min\{a,b\} \leq 3$. 
    \end{itemize}
    In particular, they are Ehrhart positive.
\end{theorem}

By the main result of Rodr\'iguez-Villegas in \cite{rodriguez-villegas}, we know that if all roots of $h_P^*(x)$ lie on the unit circle of the complex plane, then $P$ is a CL-polytope. In particular, the CL-ness of standard reflexive simplices (whose $h^*$-polynomial is $x^d+\cdots+x+1$) and cross polytopes (whose $h^*$-polynomial is $(x+1)^d$) follows. 
On the other hand, since $E_{[-1,1]^d}(x)=(2x+1)^d$, the CL-ness of $[-1,1]^d$ also directly follows. 
For the CL-ness of the other polytopes above, see Section~\ref{sec:six}. 

Let us mention that a further interesting class of CL-polytopes are certain reflexive simplices considered by Braun and Liu in \cite{braun-liu}; these belong to an even larger class of reflexive polytopes that can be used to construct pathological $h^*$-vectors; many properties of these polytopes are better expressed by arithmetic (rather than combinatorial) considerations, see \cite{braun-davis-solus}.

We recall once again that Remark~\ref{rem:unimod} guarantees that a CL-polytope of dimension $d\leq 5$ will have a unimodal $h^*$-polynomial. It is tempting to postulate that within this restricted class of reflexive polytopes the unimodality of $h^*$-polynomials might continue to hold for $d\geq 6$. The following shows that it is not the case.

\begin{theorem}
    For any $d \geq 6$ there exists a CL-polytope of dimension $d$ whose $h^*$-polynomial is non-unimodal. 
\end{theorem}

\begin{proof}
    As mentioned above, whenever the $h^*$-polynomial has all of its complex roots on the unit circle, then the Ehrhart polynomial has all of its complex roots on the critical line. In what follows, we prove the following: for any $d \geq 6$, there exists a reflexive polytope of dimension $d$ having a non-unimodal $h^*$-polynomial, but all of whose roots lie on the unit circle.
    
    We make the elementary observation that:
    \[\{bk+r : b,k,r \in \mathbb{Z}, b \geq 3, r \geq 0, k \geq r+2\}=\{ d \in \mathbb{Z} : d \geq 6\} \setminus \{7,11\},\]
    (we will address the cases $d=7$ and $d=11$ at the end).
    By \cite[Example 4.2]{payne}, we know that for any given integers $b,k,r$ with $b \geq 3, r \geq 0, k \geq r+2$, there exists a reflexive $d$-simplex, where $d=bk+r$, whose $h^*$-polynomial is the following: 
    \begin{align*}
        1+x+\cdots+x^d&+(1+\cdots+x^r)(x^k+x^{2k}+\cdots+x^{(b-1)k}) \\
        &=\underbrace{\sum_{i=0}^{k-1}x^i \cdot \sum_{j=0}^{b-1}x^{kj}+\sum_{i=0}^r x^{bk+i}}_{=1+x+\cdots+x^d}+x^k\sum_{i=0}^rx^i \cdot \sum_{j=0}^{b-2}x^{kj} \\
        &=\sum_{i=0}^{k-1}x^i \cdot \sum_{j=0}^{b-1}x^{kj}+x^k\sum_{i=0}^rx^i \cdot \sum_{j=0}^{b-1}x^{kj} \\
        &=\sum_{i=0}^{k+r}x^i \cdot \sum_{j=0}^{b-1}
x^{kj},    \end{align*}
as required. The two remaining cases $d=7$ and $d=11$ follow, respectively, by two constructions of Payne in \cite[Example~4.3 and Example~4.4]{payne}. He constructed reflexive polytopes $P$ and $Q$ of dimensions $7$ and $11$ respectively, having 
\begin{align*}
    h^*_P(x) &= 1+ 2x +6x^2+ 5x^3 +5x^4 +6x^5+ 2x^6 +x^7,\\
    h^*_Q(x) &= 1+x+4x^2+6x^3+4x^4+6x^5+6x^6+4x^7+6x^8+4x^9+x^{10}+x^{11}.
\end{align*}
We have verified using a computer that the corresponding Ehrhart polynomials have all of their complex roots lying on the critical line.
\end{proof}
\begin{remark}
    Expanding on Remark~\ref{rem:unimod}, we point out that the $h^*$-polynomials of reflexive polytopes of dimension $4$ or $5$ are not necessary log-concave. 
    In fact, we can construct examples among CL-polytopes. Consider the polytopes defined as the convex hull of the column vectors of 
    \[ \begin{bmatrix} 
    1 & 0 & 0 & 0 & -1\\
    0 & 1 & 0 & 0 & -1\\
    0 & 0 & 1 & 0 & -1\\
    0 & 0 & 0 & 1 & -2
    \end{bmatrix} \qquad\text{ and }\qquad
    \begin{bmatrix} 
    1 & 0 & 0 & 0 & 0 & -1\\
    0 & 1 & 0 & 0 & 0 & -1\\
    0 & 0 & 1 & 0 & 0 & -1\\
    0 & 0 & 0 & 1 & 0 & -2\\
    0 & 0 & 0 & 0 & 1 & -2
    \end{bmatrix}\]
    Their $h^*$-polynomials are, respectively, $x^4+x^3+2x^2+x+1$ and $x^5+x^4+2x^3+2x^2+x+1$, respectively. These are not log-concave. One can verify that the corresponding Ehrhart polynomials have all of their complex roots on the critical line.
\end{remark}
\subsection{Ehrhart positivity and \texorpdfstring{$h^*$}{h*}-real-rootedness together}

Often, one could be interested in proving that a polytope $P$ has at the same time the property of being Ehrhart positive and $h^*$-real-rooted. It is quite frequent that both of these properties indeed hold, and thus we are motivated to look for techniques that can perhaps be useful to attack both of them \emph{simultaneously}. In light of our Theorem~\ref{thm:ehr-real-rooted-implies-h*-real-rooted}, there is a scenario in which one can relate these two properties. 

The following result, discovered by Br\"and\'en in \cite{branden-transactions} in a different form, is perhaps the most useful technique that one has at disposal. 

\begin{theorem}\label{thm:magic-basis}
    Let $P$ be a lattice polytope of dimension $d$. Consider the Ehrhart polynomial written in a different basis:
        \[ E_P(x) = \sum_{i=0}^d a_i x^i (1+x)^{d-i}.\]
    If $a_0,\ldots,a_d\geq 0$, then we have simultaneously that $E_P(x)$ has non-negative coefficients and $h^*_P(x)$ is real-rooted.
\end{theorem}

\begin{proof}
    If all the $a_i$'s are non-negative, we obviously have that $E_P(x)$ has non-negative coefficients, so we focus on the real-rootedness of the $h^*$-polynomial. Fix a polynomial $Q(x)$ having degree $d$, and let us write:
        \[ Q(x) = \sum_{i=0}^d a_i x^i(1+x)^{d-i} = \sum_{i=0}^d b_i \binom{x}{i} = \sum_{i=0}^d h_i \binom{x+d-i}{d}.\]
    Brenti proved in \cite[Theorem~3.2.1]{brenti} that $\sum_{i=0}^d h_i x^i$ has only negative real-roots if and only if it happens simultaneously:
        \begin{itemize}
            \item All the $h_i$'s are non-negative.
            \item The polynomial $\sum_{i=0}^d b_i x^i$ has only negative real-roots.
        \end{itemize}
    If $Q(x)$ is an Ehrhart polynomial, then we get the first condition automatically by Stanley's non-negativity theorem \cite{stanley-hstar}, whereas the second follows from the result of Br\"and\'en in \cite[Theorem~4.2]{branden-transactions}, which says that if all the $a_i$'s are non-negative then $\sum_{i=0}^d b_i x^i$ has only negative real roots. 
\end{proof}

\begin{remark}
    Equivalent (but more complicated) versions of this result have been rediscovered several times since the paper of Br\"and\'en was published. For example, in \cite{brenti-welker}, \cite{savage-visontai} and \cite{beck-jochemko-mccullough} one can find statements that are equivalent to the above, yet they are phrased in terms of generalizations of the Eulerian polynomials. 
\end{remark}

As a consequence of the lack of a better way of referring to the property of ``expanding positively in the basis $\{x^i(1+x)^{d-i}\}_{i=0}^d$'', we often refer informally (for example in Figure~\ref{fig:flowchart}) to a polynomial $E_P(x)$ satisfying the assumptions of Theorem~\ref{thm:magic-basis} as `magic positive'. In the authors' opinion, it is undeniable that a statement such as Theorem~\ref{thm:magic-basis} is assessed quite accurately by the use of the word `magic'.

\begin{remark}
   The proof of Theorem~\ref{thm:magic-basis} can be upgraded so to say that if $f(x)$ is magic positive, then $(\mathscr{W}f)(x)$ is real-rooted, where $\mathscr{W}:\mathbb{R}[x]\to\mathbb{R}[x]$ is the operator defined in equation~\eqref{eq:wagner}. Recently, an analog of this property was established by Athanasiadis \cite[Theorem~1.1]{athanasiadis-eulerian}: he proved that if $f(x)$ is magic positive, then $(\mathscr{A}^{\circ}f)(x)$ is real-rooted where $\mathscr{A}^{\circ}:\mathbb{R}[x]\to\mathbb{R}[x]$ is the Eulerian transformation studied by Br\"and\'en and Jochemko \cite{branden-jochemko}. Yet another analog of this result was also established by Br\"and\'en and Solus for the derangement transformation \cite[Corollary~3.7]{branden-solus}.
\end{remark}

The following result provides instances in which Theorem~\ref{thm:magic-basis} applies:

\begin{theorem}
    The following families of polytopes satisfy the conditions of Theorem~\ref{thm:magic-basis} and are therefore Ehrhart positive and $h^*$-real-rooted.
    \begin{enumerate}[\normalfont(i)]
        \item Zonotopes.
        \item Pitman-Stanley polytopes.
    \end{enumerate}
\end{theorem}

The reader should not be misguided by the shortness of the list above. We have actually observed empirically that the same phenomenon seems to persist for some polytopes lying in the rich family of $\mathcal{Y}$-generalized permutohedra (but not for all of them). This family is known to be Ehrhart positive due to a result of Postnikov \cite{postnikov}, but the behavior of their $h^*$-polynomials is poorly understood. Actually, there are reasons to believe in the $h^*$-real-rootedness for \emph{all} generalized permutohedra (cf.~\cite[Conjecture~1.3]{ferroni2}) 
--- however, there is no hope of being able to use Theorem~\ref{thm:magic-basis}, because of the counterexample of Theorem~\ref{thm:matroid-negative}.

We have conducted exhaustive computations that make us believe that the following families of polytopes satisfy Theorem~\ref{thm:magic-basis}. It would be of great interest to provide actual proofs.

\begin{restatable}{problem}{magicapplication}
    For which of the following families of polytopes is the Ehrhart polynomial magic positive? 
     \begin{itemize}
        \item Partial permutohedra.
        \item Harmonic polytopes.
        \item Bipermutohedra.
        \item Fertilitopes.
        \item Generalized parking function polytopes.
    \end{itemize}
\end{restatable}

Bipermutohedra and harmonic polytopes arose in work of Ardila, Denham and Huh \cite{ardila-denham-huh}; they are the main objects of study in the papers \cite{ardila-bipermutohedron} and \cite{ardila-escobar}. We note that although fertilitopes and partial permutohedra $\mathscr{P}(m,n)$ for $n\geq m-1$ are examples of $\mathcal{Y}$-generalized permutohedra, the latter class does not always yield polytopes whose Ehrhart polynomial is magic positive; however, we speculate that it might be possible to characterize all $\mathcal{Y}$-generalized-permutohedra that do.

\section{Log-concavity of evaluations of the Ehrhart polynomial}\label{sec:five}

\subsection{Log-concavity of the Ehrhart series} Now we turn our attention to a different type of inequalities for Ehrhart polynomials. The actual property that interests us is quite related to Ehrhart positivity and $h^*$-log-concavity, yet it has enjoyed only limited attention in the literature so far. We will devote this section to address the following question: 

\begin{question}
    When is the sequence $E_P(0), E_P(1), E_P(2),\ldots$ of evaluations of the Ehrhart polynomial of $P$ an infinite log-concave sequence?
\end{question}

For short, we will refer to the property described in the above question as \emph{Ehrhart series log-concavity}. This should not be confused with the log-concavity of the Ehrhart series as a real function. 

For an \emph{arbitrary} polynomial $f(x)$ having a positive leading term, it is not difficult to prove that the inequality $f(m)^2< f(m-1)f(m+1)$ can happen for finitely many positive integer values of $m$. In other words, for a sufficiently large value of $N$, we will have $f(m)^2\geq f(m-1)f(m+1)$ whenever $m\geq N$. The reason is that, as the reader may verify, the polynomial $g(x):=f(x)^2-f(x-1)f(x+1)$ will have a positive leading term as well, and thus for sufficiently large real positive values of the variable $x$, it will be strictly positive. In particular, the point of concern in the above question is whether log-concavity breaks for ``small'' positive integers. 

As a glimpse of how subtle this property is, we offer the following result. The spirit is that the log-concavity of the Ehrhart series can fail in quite dramatic ways.

\begin{theorem}
    For every positive integer $k$ there exists a polytope $P$ of sufficiently large dimension, such that $E_P(m)^2 < E_P(m-1)E_P(m+1)$ for at least $k$ distinct positive integer values of $m$.
\end{theorem}

\begin{proof}
    Define $n_1 = 3k-2$, and $n_i = n_1 + 3(i-1)$ for each $i=2,\ldots, k$. Notice that $n_{i+1}-n_i = 3$ for each $1\leq i\leq k-1$, whereas:
        \[ n_k = n_1 + 3(k-1) = 6k-5 < 6k-1 = n_1+n_2.\]
    Preserving the notation of Example~\ref{ex:reeve-simplices}, let $P$ be the join (or free sum) of generalized Reeve simplices $\mathscr{R}_{m_1,n_1},\ldots,\mathscr{R}_{m_k,n_k}$, where the $m_i$'s are certain positive integers to be specified in what follows. Because of the condition $n_1<\cdots<n_k<n_1+n_2$, we see that the $h^*$-polynomial is of the following form: 
    \[h_{P}^*(x)=\prod_{i=1}^k(1+m_i x^{n_i})=1+m_1x^{n_1}+\cdots+m_kx^{n_k}+\sum_{i > n_k}h_i x^i.\]
    Thus, the corresponding Ehrhart polynomial looks as follows: 
    \begin{align*}
    E_P(x)=\binom{x+d}{d}+m_1\binom{x+d-n_1}{d}+\cdots+m_k\binom{x+d-n_k}{d}+\sum_{i > n_k}h_i\binom{x+d-i}{d}, 
    \end{align*}
    where $d=\sum_{i=1}^k(2n_i-1)$ denotes the dimension of the polytope $P$. 
    
    In what follows, we explain that it is possible to choose the integers $m_1,\ldots,m_k$ so that $E_P(x)$ is not log-concave at $n_i-1$ for each $i=1,\ldots,k$.
    
    By the formula, we see the following: 
    \begin{align*}
    E_P(n_i-2)&=\binom{n_i+d-2}{d}+\sum_{j=1}^{i-1}m_j\binom{n_i+d-n_j-2}{d}, \\
    E_P(n_i-1)&=\binom{n_i+d-1}{d}+\sum_{j=1}^{i-1}m_j\binom{n_i+d-n_j-1}{d}, \\
    E_P(n_i)&=\binom{n_i+d}{d}+\sum_{j=1}^{i-1}m_j\binom{n_i+d-n_j}{d}+m_i. 
    \end{align*}
    For $i=1$, we see that 
    \[E_P(n_1-1)^2-E_P(n_1-2)E_P(n_1)=\binom{n_1+d-1}{d}^2-
    \binom{n_1+d-2}{d}\left(\binom{n_1+d}{d}+m_1\right).\]
    We may take a positive integer $m_1$ which makes this negative. 
    We can do this since $n_1$ and $d$ are fixed. 
    
    Once $m_1,\ldots,m_{i-1}$ are fixed, we may follow the same steps for $m_i$. Namely, since 
    \begin{align*}
    &E_P(n_i-1)^2-E_P(n_i-2)E_P(n_i)=\left(\binom{n_i+d-1}{d}+\sum_{j=1}^{i-1}m_j\binom{n_i+d-n_j-1}{d}\right)^2 \\
    &-\left(\binom{n_i+d-2}{d}+\sum_{j=1}^{i-1}m_j\binom{n_i+d-n_j-2}{d}\right)
    \left(\binom{n_i+d}{d}+\sum_{j=1}^{i-1}m_j\binom{n_i+d-n_j}{d}+m_i\right), 
    \end{align*}
    we can choose $m_i$ in order to make this negative. 
\end{proof}

The goal now is to compare the log-concavity of the Ehrhart series with some more familiar inequalities for Ehrhart polynomials. As the following statement shows, the log-concavity of the $h^*$-polynomial is strong enough to guarantee that the sequence $E_P(0)$, $E_P(1)$, $E_P(2),\ldots$ is log-concave as well.

\begin{proposition}\label{prop:hstar-log-concave-implies-series-log-concave}
    If the $h^*$-polynomial is log-concave, then the Ehrhart series is log-concave too.
\end{proposition}

\begin{proof}
    The Ehrhart series can be written as:
        \[ \sum_{m=0}^{\infty} E_P(n) x^m = \frac{h_P^*(x)}{(1-x)^{d+1}},\]
    where $d=\dim P$. By noticing that the right-hand side is the product of the log-concave polynomial $h_P^*(x)$ with the log-concave infinite series
    \[ \frac{1}{(1-x)^{d+1}} = \sum_{m \geq 0} \binom{m+d}{m} x^m.\]
    Since the product of positive log-concave series with constant term equal to $1$ is known to be log-concave (see \cite{menon}), the sequence $E_P(0),E_P(1), E_P(2),\ldots$ is log-concave.
\end{proof}

It is quite inviting to postulate that maybe the log-concavity of the $h^*$-polynomial is \emph{equivalent} to the log-concavity of the Ehrhart series. We shall see now that it is not the case.

\begin{theorem}
    There exists a lattice polytope $P$ such that the Ehrhart series is log-concave, yet the $h^*$-polynomial is not log-concave (moreover, not even unimodal).
\end{theorem}

\begin{proof}
    Consider the Reeve tetrahedron $\mathcal{R}_1$ defined in Example~\ref{ex:reeve-tetrahedra}. Its $h^*$-polynomial is $x^2+1$, which is not unimodal. Its Ehrhart polynomial is $E_P(x) = \frac{1}{3}(x+1)(x^2+2x+3)$.
    Let us prove that for $n\geq 1$ one has: $E_P(n)^2\geq E_P(n-1)E_P(n+1)$, in turn, it suffices to ignore the factor $\frac{1}{3} (x+1)$. We therefore focus on proving that $Q(x):=x^2+2x+3$ satisfies $Q(n)^2\geq Q(n-1)Q(n+1)$. Observe that:
        \[ Q(n)^2 - Q(n-1)Q(n+1) = 2n^2 + 4n - 3 > 0\]
    for every $n\geq 1$. In particular, our polytope has a log-concave Ehrhart series.
\end{proof}

In other words, the preceding result and Proposition~\ref{prop:hstar-log-concave-implies-series-log-concave} show that the log-concavity of the $h^*$-polynomial is \emph{strictly stronger} than the log-concavity of the Ehrhart series. For the interested reader, we mention that the $h^*$-real-rootedness is equivalent to the Ehrhart series being a P\'olya frequency sequence (see for example \cite[Theorem~4.6.1]{brenti}). 

A reasonable question that one might have is whether it is possible to weaken the assumptions of Proposition~\ref{prop:hstar-log-concave-implies-series-log-concave} to just unimodality. In other words, it is tempting to ask whether having $h^*$-unimodality suffices to conclude the log-concavity of the Ehrhart series. Even though they are harder to construct, counterexamples do exist.

\begin{theorem}
    There exists a lattice polytope $P$ having unimodal $h^*$-polynomial whose Ehrhart series is not log-concave.
\end{theorem}

\begin{proof}
    Consider the $5$-dimensional polytope in $\mathbb{R}^7$ whose vertices are given by the columns of the following matrix:
    \[\left[\begin{array}{rrrrrrrrrrrr}
    1 & 1 & 1 & 1 & 0 & 0 & 0 & 0 & 0 & 0 & 0 & 0 \\
    1 & 1 & 0 & 0 & 1 & 1 & 0 & 0 & 0 & 0 & 0 & 0 \\
    0 & 0 & 0 & 0 & 1 & 1 & 1 & 1 & 0 & 0 & 0 & 0 \\
    0 & 0 & 0 & 0 & 0 & 0 & 1 & 1 & 1 & 1 & 0 & 0 \\
    0 & 0 & 0 & 0 & 0 & 0 & 0 & 0 & 1 & 1 & 1 & 1 \\
    0 & 0 & 1 & 1 & 0 & 0 & 0 & 0 & 0 & 0 & 1 & 1 \\
    111 & 112 & 0 & 1 & 0 & 1 & 0 & 1 & 0 & 1 & 0 & 1
    \end{array}\right]\]
    The $h^*$-polynomial is $h^*_P(x) = 1+6x+6x^2+113x^3$ which is unimodal (but not log-concave). The Ehrhart polynomial is given by $E_P(x) = \frac{21}{20} x^{5} + \frac{7}{8} x^{4} - 2 x^{3} + \frac{33}{8} x^{2} + \frac{139}{20} x + 1$. Moreover, we have $E_P(1)=12$, $E_P(2) = 63$ and $E_P(3) = 331$. However, observe that $E_P(1)E_P(3) = 3972 > 3969 = E_P(2)^2$, which shows that the Ehrhart series is not log-concave.
\end{proof}

At first glance, the example we constructed in the preceding proof may seem quite ad-hoc. However, it was actually chosen so to be able to conclude the following result, which combined with Proposition~\ref{prop:hstar-log-concave-implies-series-log-concave} provides a strengthening of Theorem~\ref{thm:very-ample-not-log-concave}.

\begin{corollary}
    There exists a very ample polytope $P$ whose Ehrhart series is not log-concave.
\end{corollary}

\begin{proof}
    Following the construction of Laso\'n and Micha{\l}ek \cite{lason-michalek} of a segmental fibration over the edge polytope of an even cycle, we can see that the polytope that we built in the preceding proof is very ample but not IDP.
\end{proof}

In summary, we have checked the lack of general implications in any directions between the $h^*$-unimodality and the log-concavity of the Ehrhart series. We propose the following conjecture as a counterpart of Conjecture~\ref{conj:idp-implies-unimodal}.

\begin{restatable}{conjecture}{idpehrserieslogconc}\label{conj:idp-ehr-series-logconc}
    If $P$ is an IDP lattice polytope, then the Ehrhart series of $P$ is log-concave.
\end{restatable}

Our conjecture deserves a few comments. On one hand, it is straightforward to see that the IDP property implies that for every pair of positive integers $n$ and $m$, we have that $E_P(n)E_P(m)\geq E_P(n+m)$, which for $n=m=1$ yields $E_P(1)^2\geq E_P(0) E_P(2)$ (in fact, we already used this property of IDP polytopes in the proof of Proposition~\ref{prop:idp-log-conc-three-terms}). On the other hand, we point out that the combinatorial interpretation of $E_P(n)$ as the number of lattice points in $nP$ is extremely concrete when contrasted against typical interpretations for the coefficients of the $h^*$-polynomial. For example, in the case of order polytopes, the evaluations of the Ehrhart polynomial can be interpreted as the number of order preserving maps, whereas the coefficients of the corresponding $h^*$-polynomial are slightly more complicated to describe (see \cite[Chapter~3]{stanley-ec1}). 

The careful reader may already have noticed that all the examples we have constructed so far of polytopes $P$ whose Ehrhart series is \emph{not} log-concave, share the property of also \emph{not} being Ehrhart positive. The following shows that Ehrhart positivity alone is not enough to guarantee a log-concave Ehrhart series.

\begin{theorem}
    There exists an Ehrhart positive lattice polytope whose Ehrhart series is not log-concave.
\end{theorem}

\begin{proof}
    Let $P\subseteq\mathbb{R}^3$ be the polytope having as vertices the columns of the matrix:
    \[\begin{bmatrix}
    0 & 1 & 0 & 0 & 1\\
    0 & 0 & 1 & 0 & 1\\
    0 & 0 & 0 & 1 & -13
    \end{bmatrix}\]
    Its Ehrhart polynomial is $E_P(x)=\frac{7}{3}x^3+\frac{3}{2}x^2+\frac{1}{6}x+1$.
    In particular, this is Ehrhart-positive. On the other hand, the Ehrhart series looks like $1+5x+26x^2+\cdots$, which is not log-concave.
\end{proof}

Now let us mention some instances in which the Ehrhart series is log-concave.

\begin{theorem}\label{thm:some-log-conc-ehr-series}
    The following polytopes have a log-concave Ehrhart series:
    \begin{itemize}
        \item Polytopes whose $h^*$-polynomial is log-concave.
        \item Order polytopes. 
        \item Lattice polygons.
    \end{itemize}
\end{theorem}

\begin{proof}
    The first is just a restatement of Proposition~\ref{prop:hstar-log-concave-implies-series-log-concave}. For the second, we refer to the discussion about order polytopes in Section~\ref{sec:six}. The last one follows from the fact that for \emph{every} polynomial $f(x)=ax^2+bx+1$ where $a>0$ and $b\geq 1$, one has:
    \begin{align}\label{eq:f(x)} f(x)^2 - f(x+1)f(x-1) = 2a^2\left(x^2-\tfrac{1}{2}\right)+2a(bx-1)+b^2,\end{align}
    which is strictly positive whenever $x\geq 1$. Notice that Ehrhart polynomials of lattice polygons always satisfy this assumption (in this case, recall that $b$ equals half the perimeter polygon, and therefore is certainly greater than or equal to $1$).
\end{proof}

\subsection{The log-concavity of the evaluations at negative integers}

Before finishing this article, let us focus our attention on yet another type of log-concave inequalities that an Ehrhart polynomial may satisfy, i.e., the log-concavity of the sequence $|E_P(-1)|$, $|E_P(-2)|$, $|E_P(-3)|,\ldots$. Recall that the Ehrhart-Macdonald reciprocity interprets $|E_P(-m)|$, where $m\geq 1$, as the number of lattice points lying in the \emph{relative interior} of $mP$.

By following the approach and notation of \cite[Section~9.4]{cox-little-schenck}, a general result links the Ehrhart polynomial of a lattice polytope $P$ with an Euler characteristic:
    \begin{equation} \label{eq:ehrhart-euler-char}
        E_P(\ell) = \chi(\mathscr{O}_{X_P}(\ell D_P)),
    \end{equation}
where $X_P$ is the toric variety associated to $P$, $D_P$ is the associated ample divisor, and $\mathscr{O}_{X_P}$ is the structure sheaf.

If $P\subseteq\mathbb{R}^n$ is a very ample polytope, then the associated toric variety $X_P$ admits an embedding into the projective space  $\mathbb{P}^{s-1}$ where $s = |P\cap \mathbb{Z}^n|$, by using the very ample divisor $D_P$. In particular, in this case one can define and study the so-called ``$K$-polynomial'' of $X_P$. Following \cite[Example~3.4.3]{brion} one can write the structure sheaf of $X_P$ as a (unique) linear combination of structure sheafs of linear subspaces of $\mathbb{P}^s$ of dimension at most $d=\dim P$:
    \begin{equation} \label{eq:sheafs-identity}
    \mathscr{O}_{X_P} = \sum_{i=0}^d a_i \mathscr{O}_{\mathbb{P}^i}.
    \end{equation}
The $K$-polynomial of $X_P$ is defined as: 
    \[K_{X_P}(x) := \sum_{i=0}^d a_i\,x^i.\]
Since the Euler characteristic acts linearly, we can relate equations \eqref{eq:ehrhart-euler-char} and \eqref{eq:sheafs-identity} so to write:
    \begin{equation}\label{eq:ehrhart-weird-basis}
        E_P(\ell) = \chi(\mathscr{O}_{X_P}(\ell D_P)) = \sum_{i=0}^d a_i\binom{\ell+i}{i},
    \end{equation}
where the binomial coefficients are explained by \cite[Example~9.4.1]{cox-little-schenck}. We will later show that $(-1)^{d-i} a_i\geq 0$ for all $i=0,\ldots,d$ (see the footnote~\ref{footnote-positivity} below)

In particular, for a very ample polytope $P$, the $K$-polynomial of $X_P$ encodes the same information as the Ehrhart polynomial of $P$. Some connections between $K$-polynomials with the Lorentzian property (see, e.g., \cite{castillo-cid}) have motivated the question of whether the (absolute values of the) coefficients of $K_{X_P}(x)$ form a log-concave sequence when $P$ is very ample. An elementary check reveals that:
    \[ E_P(-1) = a_0, \qquad E_P(-2) = a_0 - a_1, \qquad E_P(-3) = a_0 - 2a_1 + a_2.\]
And moreover,
    \[ E_P(-2)^2 - E_P(-1)E_P(-3) = (a_0 - a_1)^2 - a_0(a_0-2a_1+a_2) = a_1^2-a_0a_2.\]

In particular, the log-concavity of (the absolute values of) the first three coefficients of the $K$-polynomial of $X_P$ is equivalent to the log-concavity of (the absolute values of) the first three negative evaluations of the Ehrhart polynomial of $P$. For an arbitrary polytope $P$ (not necessarily very ample) we define the \emph{$K$-polynomial of $P$}, denoted $K_P(x)$, by
    \[ K_P(x) := \sum_{i=0}^d a_i\, x^i, \qquad \text{ whenever } \qquad E_P(x) = \sum_{i=0}^d a_i \binom{x+i}{i}. \]
This definition guarantees that if $P$ is very ample, we have $K_{X_P}(x) = K_P(x)$.

We will refer to the sequence $|E_P(-1)|$, $|E_P(-2)|$, $|E_P(-3)|$, $|E_P(-4)|$, $\ldots$ as the \emph{interior Ehrhart series} of $P$.

\begin{proposition}
    The log-concavity of the absolute values of the coefficients of $K_{P}(x)$ implies the log-concavity of the interior Ehrhart series of $P$.
\end{proposition}

\begin{proof}
    We have that
        \[ E_P(-x) = \sum_{i=0}^d a_i \binom{-x+i}{i} = \sum_{i=0}^d (-1)^i a_i \binom{x-1}{i}.\]
    In particular, the polynomial $\widetilde{E}_P(x) := (-1)^{d}E_P(-x-1)$ satisfies:
        \[ \widetilde{E}_P(x) = \sum_{i=0}^d \widetilde{a}_i \binom{x}{i}\]
    and we know that $\widetilde{a}_i := (-1)^{d-i} a_i \geq 0$. The log-concavity\footnote{Recall that our convention is that log-concavity also affirms that the numbers in the sequence must be strictly positive.} of $(\widetilde{a}_0,\ldots,\widetilde{a}_d)$ implies by an old result attributed to Karlin (see \cite[Theorem~2.5.7]{brenti}) that the infinite sequence $\widetilde{E}_P(0), \widetilde{E}_P(1), \widetilde{E}_P(2), \ldots$ is log-concave. In particular, this tells that the infinite sequence of evaluations at negative integers, $|E_P(-1)|,| E_P(-2)|, |E_P(-3)|$, \ldots\, is log-concave.
\end{proof}

The following result shows that for polytopes that are not very ample, the log-concavity of the interior Ehrhart series of $P$ may fail.

\begin{theorem}\label{thm:negative-eval}
    There exists a lattice polytope $P\subseteq \mathbb{R}^4$ of dimension $4$ such that $|E_P(-2)|^2 < |E_P(-1)|\cdot|E_P(-3)|$. 
\end{theorem}

\begin{proof}
    Let $P \subseteq \mathbb{R}^4$ be the polytope having the vertices as vertices the columns of the matrix: 
    \[\begin{bmatrix}
        1    &0    &0    &1   &-2 \\
        0    &1    &0    &2   &-3 \\
        0    &0    &1    &3   &-4 \\
        0    &0    &0    &5   &-5
    \end{bmatrix}.\]
    Then this is a (non spanning) reflexive polytope and has the Ehrhart polynomial \[ E_P(x)=\tfrac{25}{24}x^4+\tfrac{25}{12}x^3+\tfrac{35}{24}x^2+\tfrac{5}{12}x+1.\]
    In fact, the evaluations of negative integers are $E_P(-1)=1, E_P(-2)=6, E_P(-3)=41, E_P(-4)=156$, etc. 
\end{proof}

Petter Br\"and\'en asked at the ``Ehrhart day'' meeting held at KTH (Stockholm) in June 2023 for classes of polytopes satisfying this log-concavity property. We have already set enough motivation to pose the following question, proposing the class of very ample polytopes as a reasonable departing point.

\begin{restatable}{question}{lastquestion}\label{question:very-ample-negative-series}
\leavevmode
    \begin{enumerate}[(a)]
        \item Does there exist a very ample polytope $P$ for which $|E_P(-2)|^2 < |E_P(-1)|\cdot|E_P(-3)|$?
        \item Does there exist a very ample polytope whose interior Ehrhart series is not log-concave?
    \end{enumerate}
\end{restatable}

A negative answer to any of these two questions would disprove the log-concavity question for the coefficients of $K$-polynomials of $X_P$ for very ample polytopes.

We now show a subtle result, which connects (in our opinion unexpectedly) the log-concavity of the $h^*$-polynomial with the log-concavity of the interior Ehrhart series. 

\begin{proposition}
    If the $h^*$-polynomial of $P$ is log-concave, then so is the interior Ehrhart series of $P$.
\end{proposition}

\begin{proof}
    Let us write $E_P(x) = \sum_{j=0}^d a_j \binom{x+j}{j}$. Computing the Ehrhart series of $P$ yields:
    \begin{align*}
        \sum_{m\geq 0} E_P(m) z^m &= \sum_{m\geq 0} \sum_{j=0}^d a_j \binom{m+j}{j}  z^m
        = \sum_{j=0}^d a_j \sum_{m\geq 0}\binom{m+j}{j} z^m
        = \sum_{j=0}^d a_j \frac{1}{(1-z)^{j+1}}\\
        &= \frac{1}{(1-z)^{d+1}} \sum_{j=0}^d a_j (1-z)^{d-j}.
    \end{align*}
    In particular, this tells us that $h^*_P(x) = \sum_{j=0}^d a_j (1-x)^{d-j}$ or, more succinctly,
    \[ h^*_P(x+1) = \sum_{j=0}^d (-1)^{d-j}a_j \, x^{d-j}.\]
    In other words, the coefficients $(a_0,\ldots,a_d)$ are (up to a sign) the coefficients of $h^*_P(x+1)$ in reversed order\footnote{This shows that $(-1)^{d-j} a_j \geq 0$ as we claimed earlier.\label{footnote-positivity}}. In particular, the assumption on the log-concavity of the coefficients of $h^*_P(x)$ yields by a classical result (see, e.g., \cite[Corollary~8.4]{brenti-update}) the log-concavity of the coefficients of $h^*_P(x+1)$. Hence, the absolute values of $a_0,\ldots,a_d$ form a log-concave sequence, and by the reasoning above the statement of Theorem~\ref{thm:negative-eval} this gives the log-concavity of the interior Ehrhart series of $P$.
\end{proof}

We finish this article with a positive result towards Br\"and\'en's questions. We omit the proof, because it is very similar to the proof of Theorem~\ref{thm:some-log-conc-ehr-series}.

\begin{theorem}
    If $P$ is a lattice polygon, then its interior Ehrhart series is log-concave.
\end{theorem}

\section{Miscellany}\label{sec:six}

Throughout this article we have mentioned several families of polytopes appearing in combinatorics, and we have posed several questions, problems, and new conjectures. We give an overview of these families of polytopes, and we collect all the problems and conjectures altogether.

\subsection{A miscellany of combinatorial polytopes}

\begin{footnotesize}

\subsubsection*{Generalized permutohedra} A generalized permutohedron \cite{postnikov} is a polytope all of whose edges are parallel to some vector of the form $e_i-e_j$. The class of (integral) generalized permutohedra coincides with base polytopes of discrete polymatroids. They were conjectured by Castillo and Liu in \cite{castillo-liu} to be Ehrhart positive, but this was disproved in \cite{ferroni3}. They are known to be IDP \cite{mcdiarmid} and have a regular unimodular triangulation \cite{backman-liu}, but are conjectured to be even higher in the hierarchy (see Conjecture~\ref{conj:gen-perm-quad-triang}). 

\subsubsection*{$\mathcal{Y}$-generalized permutohedra} A $\mathcal{Y}$-generalized permutohedron is a lattice polytope which arises as a Minkowski sum of dilations of standard simplices. Not all generalized permutohedra are of this form. This family is essentially disjoint from the class of matroids. Postnikov showed in \cite{postnikov} that all these polytopes are Ehrhart positive. Higashitani and Ohsugi proved in \cite{higashitani-ohsugi} that they possess regular unimodular triangulations. Furthermore, Backman and Santos (personal communication) proved that these polytopes have quadratic triangulations. Up to possibly a unimodular equivalence, many well-studied families of polytopes fall into this class, for example nestohedra, graph associahedra \cite{postnikov}, fertilitopes \cite{defant}, partial permutohedra $\mathscr{P}(m,n)$ for $n\geq m-1$ \cite{partial-permutohedra}, and Pitman-Stanley polytopes \cite{pitman-stanley}.

\subsubsection*{Matroid polytopes} Given a matroid $\M$ having ground set $E$ and set of bases $\mathscr{B}$, one can construct the so-called base polytope of $\M$ as the polytope in $\mathbb{R}^E$ having as vertices the indicator vectors of the bases $B\in\mathscr{B}$, i.e., as the convex hull of the vectors $e_B:=\sum_{i\in B} e_i$ for each $B\in \mathscr{B}$. Matroid polytopes are exactly the class of $0/1$-generalized permutohedra. They were conjectured by De~Loera, Haws and K\"oppe to be Ehrhart positive, but this was disproved by Ferroni in \cite{ferroni3}. They are known to be IDP \cite{mcdiarmid,michalek-sturmfels} and have regular unimodular triangulations \cite{backman-liu}. It is conjectured that they are $h^*$-real-rooted \cite{ferroni2} --- the case of $h^*$-unimodality, which would follow from Conjecture~\ref{conj:idp-implies-unimodal}, is open as well, and is conjectured in \cite{deloera-haws-koppe} too.

\subsubsection*{Alcoved polytopes} An alcoved polytope (of type A) \cite{lam-postnikov} is a polytope whose supporting inequalities are of the form $\alpha_{ij}\leq x_j-x_i \leq \beta_{ij}$. It is customary to apply the change of coordinates $x_j := y_1+y_2+\cdots+y_j$ for each $j=1,\ldots,n$, so that a polytope whose supporting inequalities  are of the form $\alpha_{ij} \leq y_{i+1}+y_{i+2}+\cdots+y_j \leq \beta_{ij}$ is also called alcoved. Alcoved polytopes possess quadratic triangulations due to a result of Payne \cite{payne-koszul} (see also \cite{haase-paffenholz-piechnik-santos}). The class of alcoved polytopes encompasses a number of well-studied families of polytopes such as Lipschitz polytopes \cite{sanyal-stump}, positroids \cite{ardila-rincon-williams}, polypositroids \cite{lam-postnikov-polypositroids} and order polytopes \cite{stanley-poset}. The question on $h^*$-unimodality for alcoved polytopes was explicitly considered in \cite{sinn-sjoberg}, where a partial result was achieved.

\subsubsection*{Polypositroids} A polypositroid \cite{lam-postnikov-polypositroids} is a polytope that is at the same time alcoved and a generalized permutohedron. 

\subsubsection*{Positroids} A positroid is a polytope that is at the same time alcoved and a matroid polytope. It coincides with the class of $0/1$-polypositroids. The underlying matroid of a positroid possesses several distinguished combinatorial features. Many classes of matroids are known to be positroids (up to isomorphisms), for example all rank $2$ matroids, all uniform matroids (hence, hypersimplices), all lattice path matroids, and all series-parallel matroids. Positroids are conjectured to be Ehrhart positive \cite{ferroni-jochemko-schroter}.

\subsubsection*{Shard polytopes} Shard polytopes were introduced in \cite{padrol-pilaud-ritter}. They are matroid polytopes whose corresponding matroid is series-parallel. In particular, it follows that they are isomorphic to a positroid, and hence possess quadratic triangulations. As special cases of positroids and matroids, they are conjectured to be Ehrhart positive and $h^*$-real-rooted, respectively.

\subsubsection*{Hypersimplices} Hypersimplices are the base polytopes of uniform matroids. They possess quadratic triangulations. They are alcoved, and thus positroids. They are Ehrhart positive \cite{ferroni1} (see also \cite{ferroni-mcginnis,mcginnis}), but it is an open problem to prove that they are $h^*$-real-rooted (or even $h^*$-unimodal).

\subsubsection*{Order polytopes} The order polytope of a finite poset $P$ on $[n]$ is defined by the intersection of the unit cube $[0,1]^n$ with the inequalities $x_i\leq x_j$ for every pair of elements $i\prec j$ in $P$. These polytopes were studied by Stanley in \cite{stanley-poset}, and their Ehrhart theory is quite related with the order polynomial of the poset and the theory of $P$-partitions (see \cite[Chapter~3]{stanley-ec1}). A special case of a conjecture due to Neggers and Stanley asserted that the $h^*$-polynomial of order polytopes were real-rooted. This was disproved by Stembridge in \cite{stembridge}, motivated by the work of Br\"and\'en \cite{branden-neggers}. Order polytopes are $0/1$-polytopes, are alcoved and possess quadratic triangulations. They fail to be Ehrhart positive in general \cite{liu-tsuchiya}. In the Gorenstein case, they are known to be $\gamma$-positive \cite{branden-gamma}. It is an open problem to decide whether Gorenstein order polytopes are $h^*$-real-rooted (or even log-concave). For general order polytopes, it is open to prove that they are $h^*$-log-concave (or even unimodal, as a special case of Conjecture~\ref{conj:idp-implies-unimodal}). It was proved by Brenti in \cite[Theorem~7.6.5]{brenti} that the Ehrhart series of order polytopes are log-concave (an alternative proof and a strengthening were independently found by Chan, Pak and Panova in \cite[Theorem~1.3]{chan-pak-panova}).

\subsubsection*{Edge polytopes}
Given a finite graph $G$ on $[d]$ with the edge set $E(G)$, the edge polytope is defined to be the convex hull of $\{e_i+e_j : \{i,j\} \in E(G)\}$. 
This was explicitly introduced in \cite{ohsugi-hibi-1998}, but the algebraic counterpart of the edge polytope, known as \textit{edge ring} of $G$, already appeared in \cite{simis-vasconcelos-villarreal-1994}. 
Notice that the edge polytope of the complete graph is nothing but the second hypersimplex, and that edge polytopes of complete multipartite graph coincide with matroid polytopes of loopless matroids of rank $2$ \cite{ferroni-jochemko-schroter}.
Many combinatorial/algebraic properties on edge polytopes/edge rings have been studied. See \cite[Section~5]{herzog-hibi-ohsugi}. We point out that in some cases edge polytopes have quadratic triangulations, such as simple edge polytopes \cite{matsui-hibi-higashitani-nagazawa-hibi}, bipartite graphs that do not contain chordless cycles of length greater than five \cite{ohsugi-hibi-bipartite}, and complete multipartite graphs \cite{ohsugi-hibi-2000} (this also follows from the fact that rank 2 matroids are positroids). 

\subsubsection*{Symmetric edge polytopes}
Given a finite simple graph $G$ on $[d]$ with the edge set $E(G)$, the symmetric edge polytope is defined to be the convex hull of $\{\pm(e_i-e_j) : \{i,j\} \in E(G)\}$. 
This was introduced in \cite{matsui-hibi-higashitani-nagazawa-hibi} in the context of the investigation of CL-polytopes. 
It is known that symmetric edge polytopes are reflexive polytopes (\cite[Proposition 4.2]{matsui-hibi-higashitani-nagazawa-hibi}) and they possess regular unimodular triangulations \cite{higashitani-jochemko-michalek}, in particular they are $h^*$-unimodal. It is known that in general they are neither $h^*$-real-rooted nor possess quadratic triangulations (as an example take a cycle of length $7$). However, a conjecture of Ohsugi and Tsuchiya postulates that they are $\gamma$-positive \cite[Conjecture~5.11]{ohsugi-tsuchiya-2021}); evidence for some classes of graphs are confirmed to be $\gamma$-positive (see \cite{higashitani-jochemko-michalek} and \cite{ohsugi-tsuchiya-2021}), further evidence is provided for the positivity of some specific coefficients of the $\gamma$-polynomial (see \cite{dali-juhnke-venturello}).
The CL-ness of symmetric edge polytopes of complete bipartite graphs for $K_{a,b}$ with $\min\{a,b\}\leq 3$ is proved in \cite[Section~4]{higashitani-kummer-michalek}. A generalization of these polytopes to all \emph{regular matroids} is proposed in \cite{dali-juhnke-koch}.

\subsubsection*{Root polytopes}
Root polytopes are defined to be the convex hull of a given root system, e.g., type A, B, C, D. 
We usually regard root polytopes as lattice polytopes with respect to the underlying root lattice. 
It is known that root polytopes of type A, C and D are reflexive polytopes, while those of type B are not. 
In \cite[Theorem 2]{ardila-beck-hosten-pfeifle-seashore}, the $h^*$-polynomials of root polytopes of type A,C and D are given, respectively. 
We also notice that the root polytope of type A coincides with the symmetric edge polytope of complete graphs. 
Regarding the CL-ness of root polytopes, those of type A are proved to be CL in \cite[Example 3.5]{higashitani-kummer-michalek} as well as for the CL-ness of their dual in \cite[Corollary 5.4]{higashitani-kummer-michalek}. 
Root polytopes of type C are also proved to be CL in \cite[Example 3.6]{higashitani-kummer-michalek}, while the CL-ness of their dual is proved in \cite[Theorem 1.2]{higashitani-yamada}. 
We also mention that the CL-ness of the dual of the Stasheff polytope is also provided in \cite[Example 3.4]{higashitani-kummer-michalek}. 

\subsubsection*{Pitman-Stanley polytopes} These polytopes were introduced in \cite{pitman-stanley}. They are alcoved and can be seen as part of the family of $\mathcal{Y}$-generalized permutohedra via a unimodular equivalence, which arises via lifting all the vertices to the same height (see \cite[Example~3.1.6]{liu-survey}). In particular, it follows that they are (up to a unimodular equivalence) a special type of polypositroid. Since they are alcoved, these polytopes have quadratic triangulations. Since they are $\mathcal{Y}$-generalized permutohedra, they are Ehrhart positive. In \cite{ferroni-morales} it is proved that their $h^*$-polynomials is real-rooted. For a particular choice of parameters, one may obtain the weakly increasing parking function polytope of \cite{vindas}, denoted $\mathcal{X}_n^w(a,b)$ (its Ehrhart polynomial is real-rooted by \cite[Corollary~4.4]{vindas}).

\subsubsection*{Zonotopes} These polytopes arise as a Minkowski sum of line segments each of which has the origin as an endpoint. We refer to \cite[Chapter~9]{beck-robins} for the proofs of two classical facts about zonotopes: they are Ehrhart positive and are IDP (because they can be paved using parallelotopes). The $h^*$-real-rootedness of zonotopes was established by Beck, Jochemko and McCullough \cite{beck-jochemko-mccullough}, and their proof is essentially equivalent to proving that Ehrhart polynomials of zonotopes expand positively in the magic basis (see Theorem~\ref{thm:magic-basis}). Perhaps a little embarrasingly, it is not known whether zonotopes admit unimodular triangulations in general (cf. \cite[Section~1.5.4]{haase-paffenholz-piechnik-santos}).

\subsubsection*{Compressed polytopes} A polytope is said to be \emph{compressed} (or $2$-level) if all of its weak pulling triangulations are unimodular. Examples of compressed polytopes include Birkhoff polytopes \cite{athanasiadis-birkhoff}, some Gelfand-Testlin polytopes \cite[Corollary~27]{alexandersson}, order polytopes \cite{stanley-poset}, hypersimplices \cite{sturmfels} and edge polytopes of complete multipartite graphs \cite{ohsugi-hibi-2000}

\subsection{A summary of problems and conjectures}

For the readers interested only in the open problems and conjectures mentioned throughout the article, we collect all of them here.

\idpunimodal*
\products*
\veryample*
\genperquadratic*
\galehrhart*
\negativepatterns*
\positroidsehrhartpositive*
\edgesymedgeehrpos*
\magicapplication*
\idpehrserieslogconc*
\lastquestion*

\end{footnotesize}

\section*{Acknowledgments}

We thank the American Institute of Mathematics and the organizers of the workshop
``Ehrhart polynomials: inequalities and extremal constructions'' held in May 2022, as this project
was initiated during that workshop. We are especially grateful to Petter Br\"and\'en for numerous stimulating conversations and suggestions regarding inequalities for polynomials, and to Matt Larson for sharing many insightful thoughts about the connection between $K$-polynomials and Ehrhart theory. We thank Akiyoshi Tsuchiya for bringing to our attention the example provided in the proof of Theorem~\ref{thm:non-realrooted} and thank Gabriele Balletti for suggesting us an idea to construct the example of Theorem~\ref{thm:negative-eval}. We benefited from useful discussions and comments by Per Alexandersson, Christos Athanasiadis, Matthias Beck, Roger Behrend, Federico Castillo, Alessio D'Al\`i, and Andrew Sack. We also acknowledge a careful reading and useful suggestions by two anonymous referees.

\bibliographystyle{amsalpha0}
\bibliography{bibliography}

\end{document}